\documentclass[smallextended]{svjour3}       
\smartqed  
\usepackage{graphicx}
 \usepackage{mathptmx}      
\usepackage{graphicx}
\usepackage{latexsym}
\usepackage{amsmath}
\usepackage{amssymb}
\usepackage{amsfonts}

\newtheorem{assumption}{Assumption}

\usepackage{booktabs}
\usepackage{multirow}
\usepackage{hyperref}
\usepackage{subfig}
\usepackage{xcolor}
\usepackage{comment}

\usepackage{booktabs}
\usepackage{multirow}
\usepackage{hyperref}
\usepackage{subfig}
\usepackage{xcolor}

\usepackage{comment}
 \journalname{}
\begin{document}

\title{Partition-based Distributionally Robust Optimization via Optimal Transport  with Order Cone Constraints
\thanks{This project has received funding from the European Research Council (ERC) under the European Union's Horizon 2020 research and innovation programme (grant agreement no. 755705). This work was also supported in part by the Spanish Ministry of Economy, Industry and Competitiveness and the European Regional Development Fund (ERDF) through project ENE2017-83775-P.}
}

\titlerunning{Partition-based DRO via Optimal Transport  with Order Cone Constraints}        

\author{Adri\'an Esteban-P\'erez        \and
        Juan M. Morales 
}


\institute{Adri\'an Esteban-P\'erez \at
              Department of Applied Mathematics,  University of M\'alaga, M\'alaga, 29071, Spain \\
              \email{adrianesteban@uma.es}           
           \and
           Juan M. Morales  \at
        Department of Applied Mathematics,  University of M\'alaga, M\'alaga, 29071, Spain \\
              \email{juan.morales@uma.es}           
}

\date{Received: date / Accepted: date}

\maketitle

\begin{abstract}
In this paper we wish to tackle stochastic programs affected by ambiguity about the probability law that governs their uncertain parameters. Using optimal transport theory, we construct an ambiguity set that exploits the knowledge about the distribution of the uncertain parameters, which is  provided by: i) sample data and ii) \emph{a-priori} information on the order among the probabilities that the true data-generating distribution  assigns to some regions of its support set.  This type of order is enforced by means of order cone constraints and can encode a wide range of information on the shape of the probability distribution of the uncertain parameters such as information related to monotonicity or multi-modality. We seek decisions that  are distributionally robust. In a number of practical cases, the resulting distributionally robust optimization (DRO) problem can be  reformulated as a finite convex problem  where the  \emph{a-priori} information  translates into linear constraints. In addition, our method inherits the finite-sample performance guarantees of the Wasserstein-metric-based DRO approach proposed by Mohajerin Esfahani and Kuhn (2018), while generalizing this and other popular DRO approaches. Finally, we have designed numerical experiments  to analyze the performance of our approach
with the newsvendor problem and the problem of a strategic firm competing \`{a} la Cournot in a market.
\end{abstract}

\section{Introduction}
\emph{Distributionally robust optimization}  (DRO) is a powerful modeling framework
for optimization under uncertainty that emerges from considering that the probability distribution of the  problem's uncertain parameters is, in itself, also uncertain. This gives rise to the notion of the \emph{ambiguity set}, that is, a set where the modeler assumes that the true distribution of the  problem's uncertain parameters is contained. The goal of DRO is therefore to find the decision maker's choice that is optimal against the worst-case probability distribution within the prescribed ambiguity set. Hence, DRO can be seen as a \emph{marriage} between stochastic programming and robust optimization, working with probability distributions as the former does, while hedging the decision-maker against the worst case as the latter typically aims to do. Since the work of \cite{Scarf1958}, many DRO models have been proposed and studied in the technical literature, especially over the last decade, in which DRO has attracted a lot of attention and become very popular in the field of optimization under uncertainty as an alternative to other paradigms.
We refer the reader to \cite{Keith2021,Rahimian2019} for recent surveys on DRO and optimization under uncertainty.
Naturally, the construction of the ambiguity set is key to the practical performance of DRO. It is no wonder, therefore, that much effort has been applied to this issue, resulting in several ways to specify and characterize the ambiguity set, namely:
\begin{enumerate}
    \item \emph{Moment-based approach:} The
    ambiguity
    set is defined as the set of all
    probability distributions whose moments  satisfy certain constraints; see,  \cite{Delage2010a,Gao2017a,Liu2018,Liu2019b,Mehrotra2014,Nakao2017a,Xin2021,Zymler2013}, to name a few.
    \item \emph{Dissimilarity-based approach:}  The
    ambiguity
    set is defined as the set of all probability distributions whose \emph{dissimilarity} to a prescribed distribution (often referred to as the \emph{nominal distribution}) is lower than or equal to a given value. Within this category, the choice of the \emph{dissimilarity} function leads to a wealth of distinct variants:
    \begin{enumerate}
        \item \emph{Optimal-transport-based (OTP) approach:} Here, we include the work in   \cite{Blanchet2021,Blanchet2019proc,Gao2016,MohajerinEsfahani2018,JMLRShafieezadeh-Abadeh},   among many others, all of which use, as the dissimilarity function, the well-known \emph{Wasserstein} distance, which exhibits some nice statistical convergence properties. Our work is also based on optimal mass transportation and consequently, it  falls within this category.
        \item \emph{$\phi$-divergences-based approach:} This class comprises all those approaches, which use
    $\phi$-divergences (such as the Kullback-Leibler divergence), for instance,
     \cite{Bayraksan2015,Ben-Tal2013,Namkoong2016}. We also include in this group the likelihood-based approaches, proposed by \cite{Duchi2021} and \cite{Wang2016}.
        \item \emph{Other measures of dissimilarity:} This category includes all
        other dissimilarity-based procedures for constructing ambiguity sets  than those already mentioned, such as those that utilize the family of $\zeta$-structure
        probability metrics (for example, the total variation metric, the Bounded
        Lipschitz metric ...), see, for example, the work in \cite{Rahimian2019mapr} and
       \cite{Zhao2015}, and the Prokhorov metric \cite{Erdogan2006}.
     \end{enumerate}
 \item \emph{Hypothesis-test-based approach:} The ambiguity set is made up of all those probability distributions which, given a data sample, pass a certain hypothesis test with a prescribed confidence level; see, for example, the work in \cite{Bertsimas2018a,Bertsimas2018b,Chen2019}.

\end{enumerate}


In the work we present here, we focus on ambiguity sets that are formulated by way of an optimal mass transportation problem. In fact, when the cost function in this problem is a metric, we recover the Wasserstein metric, which is indeed a metric for probability measures. According to \cite{Blanchet2021,Gao2016,MohajerinEsfahani2018}, the Wasserstein metric has  nice and interesting properties which make it a good choice in DRO,  compared to popular alternative choices such as $\phi$-divergences (see Sections 1.1 and 5.1 in \cite{Gao2016}, and the Introduction in \cite{MohajerinEsfahani2018} for a comparative analysis). Interestingly, the Wassertein distance offers a powerful theoretical framework to establish rates and guarantees of convergence. Furthermore, the conservatism implied by the ambiguity sets, built by means of the Wasserstein metric can be easily controlled, based on those rates.

Other cost functions can be used in the optimal mass transportation problem, but these do not generally result in a metric, which, most likely,  makes it much harder to establish rates of convergence and theoretical guarantees.

However, one disadvantage of using the Wasserstein metric is that the worst-case probability distribution may take the form of a Dirac distribution \cite{Yue2020}, which is implausible in practice. Ambiguity sets containing unrealistic distributions may result in overly conservative solutions, since  protection against these implausible distributions may require a decision that is more expensive than actually needed. Consequently, ambiguity sets that are solely based on a Wasserstein ball may lead to excessively costly solutions. In order to reduce the degree of conservatism, the authors in \cite{Gao2017a,Liu2021scheme,Wang2018,Yao2018} consider ambiguity sets that are formulated using the Wasserstein metric in conjunction with moment constraints.
Specifying these constraints, however, requires the estimation of the relevant parameters. Moreover, adding second-order moment information leads to semidefinite programs. In fact, as underlined in \cite{Liu2021scheme}, although the mixture of moment conditions and the Wasserstein metric allows the decision maker to  exclude pathological distributions and results in good out-of-sample performance, only in some special cases, e.g., when the objective function is piecewise linear with respect to the uncertain parameter, can the DRO problem  be reformulated as a tractable semidefinite program. For this reason, they propose a method to approximate the solution of DRO problems with ambiguity sets that are based on both moment conditions and the Wasserstein metric.

Our work follows the path of the work in \cite{Gao2017a,Liu2021scheme,Wang2018,Yao2018}: In an attempt to avoid overly conservative solutions, we seek to enrich the specification of Wasserstein ambiguity sets with \emph{a-priori} information on the true probability distribution of the problem's uncertain  parameters. Nonetheless, unlike the aforementioned approaches, we represent this information in the form of order cone constraints on the probability masses associated with a partition of the sample space. This has the advantage that the inclusion of such \emph{a-priori} information does not jeopardize the computational tractability of the underlying mathematical program. Our main contributions can be summarized as follows:
\begin{enumerate}
    \item In real-world decision-making problems, it is common to count on qualitative and expert information conveying some sense of order between the probabilities of occurrence of certain events.  For instance, in the multi-item newsvendor problem, the experienced decision maker may state that high demand values for a certain item are \emph{more likely} to occur than low ones.  This can occur, for example, when the true data-generating probability distribution is known or believed to be a mixture of distributions. In this case, determining the number of partitions in our DRO approach would be equivalent to estimating the number of components of the mixture. Indeed, one should use a number of partitions close to the number of distributions in the mixture. The task of inferring that number and the contribution of each component to the mixture is a relevant and well-known problem in statistics, which falls within the so-called realm of \emph{finite mixture models} (see,  Chapter 6 of \cite{McLachlan2000}).  In our approach, however, we assume that part of this inference task has already been done and so some of the inference results are available to the decision maker. Our aim is to exploit this type of qualitative information in the construction of the ambiguity set.  Most importantly, our DRO approach  protects the decision maker against the ambiguity in this inference process. For this purpose, we propose partitioning the support of the random parameter vector and bestow a partial order on (some of) the probability masses of the resulting subregions. This partial order can be described by a graph, which, in turn, can be associated with a convex cone. Consequently, the partial order can be embedded into the formulation of the ambiguity set in the form of conic constraints. The use of these types of cones is well known in the field of statistical inference with order restrictions (see,  \cite{NEMETH201680,Silvapulle2011}).

    \item As shown in the numerical tests, this partial order can be leveraged, among other things, to easily encode multi-modality using \emph{linear} constraints, as opposed to other approaches based on semidefinite programming  (see, for example, the work in \cite{Hanasusanto2015}), with the consequent benefit in terms of computational complexity. The recent papers \cite{Chen2019,Lam2017,Li2019mapr} consider ambiguity
sets with moment and generalized unimodal constraints. Our approach, however, can practically model a wider range of ``shapes'' beyond unimodality (see Subsection \ref{order_conic_section} for more details). 
    \item In addition to the order cone constraints on the probability masses linked to the different subregions
    of the partitioned sample space, these probability masses can  also be treated as random,
    with their probability distribution belonging to a certain ambiguity set. This way, our modeling framework extends the
    two popular DRO paradigms proposed by \cite{MohajerinEsfahani2018}, and \cite{Bayraksan2015,Ben-Tal2013}, respectively. Indeed,


    \begin{itemize}
        \item If we consider one partition only, that is, the entire sample space itself, there is no uncertainty about the associated probability mass (which is, evidently, equal to one) and no partial order can be established. If we now use a distance as the transportation cost function, our DRO framework reduces to that of \cite{MohajerinEsfahani2018}.
        \item On the contrary, in order to get the DRO framework of \cite{Bayraksan2015,Ben-Tal2013}, we just need to i) consider a number of partitions such that every partition contains a single data point from the sample, ii) assume that the distribution of their probability masses belongs to a $\phi$-divergence-based ambiguity set and iii) ignore any other information on the true probability distribution of the problem's uncertain  parameters (namely, partial order and ambiguity in the conditional distributions).

    \end{itemize}
    For their part, the authors in \cite{Chensim2020} have proposed  a different ambiguity set that also covers these two DRO approaches as special cases. However, their ambiguity set does not include the DRO framework we propose, as we note later.

    \item Under mild assumptions, we provide a tractable reformulation of our proposed DRO framework and show that it enjoys finite sample and asymptotic consistency guarantees.
    \item Finally, we numerically illustrate  the benefits in having \emph{a-priori} information by comparing our DRO framework with the well-known sample average approximation (SAA) solution and the Wasserstein metric-based approach of \cite{MohajerinEsfahani2018}. To this end, we consider the single and multi-item newsvendor problems and the problem of a strategic firm competing \`{a} la Cournot in a market.
\end{enumerate}

The rest of the paper is organized as follows.  Section \ref{ddro_section} includes some preliminaries to the optimal transport problem, we formulate the proposed DRO approach and present tractable reformulations. Convergence properties and performance guarantees are theoretically discussed in Section \ref{theory_section}. Section \ref{numerics} provides the results from numerical experiments. Finally, Section \ref{conclusions} concludes the paper.

\textbf{Notation}.
We  use $\overline{\mathbb{R}}$ to denote
the extended real line, and adopt the
conventions of its associated arithmetic. Furthermore, $\mathbb{R}_+$
denotes the set of non-negative real numbers.
We employ lower-case bold face
letters to represent vectors and bold face capital letters for matrices. We use $\textrm{diag}(a_1,\ldots, a_m)$ for a diagonal matrix of size $m \times m$ whose diagonal elements are equal to
$a_1, \ldots, a_m$. Moreover, given a matrix $\mathbf{M}$, its  transpose matrix will be written as $\mathbf{M}^T$.
We define
$\mathbf{e}$ as  the array with all its
components equal to $1$. The inner
product of two vectors $ \mathbf{u},
\mathbf{v} $ (in a certain space) is denoted  $\langle
\mathbf{u}, \mathbf{v}\rangle = \mathbf{u}^T
\mathbf{v}$.
Given any norm $\left\| \cdot
\right\|$ in the Euclidean space (of a given
dimension $d$), the dual norm is defined as
$\left\| \mathbf{u}
\right\|_{*}=\sup_{\left\| \mathbf{v}
\right\| \leqslant 1} \langle
\mathbf{u},\mathbf{v} \rangle$. Given a function $f: \mathbb{R}^d \rightarrow \overline{\mathbb{R}}$, we will say that $f$ is a proper function if $f(\mathbf{x}) < +\infty$ for at least one $\mathbf{x}$ and $f(\mathbf{x}) > -\infty$ for all $\mathbf{x} \in \mathbb{R}^d$. Additionally, the convex conjugate function of $f$, $f^*$, is defined as
$f^*(\mathbf{y}):= \sup_{\mathbf{x} \in \mathbb{R}^d} \langle \mathbf{y},\mathbf{x} \rangle -f(\mathbf{x})$. It is well known that if $f$ is a proper function, then $f^*$ is also a proper function.
Given a set $A \subseteq \mathbb{R}^d$, we  denote its relative interior as $\textrm{relint}(A)$. Similarly, we  refer to its interior as $\textrm{int}(A)$.
The support function of set $A$, $S_A$, is
defined as
$S_A(\mathbf{b}):= \sup_{\mathbf{a} \in A}
\langle \mathbf{b}, \mathbf{a} \rangle$. The dual cone $\mathcal{C}^*$ of a cone $\mathcal{C}$ is given by $\mathcal{C}^*:=\{ \mathbf{y} \; /\;  \langle \mathbf{y}, \mathbf{x} \rangle \geqslant 0, \; \forall \mathbf{x} \in \mathcal{C} \}$.
 We use the symbol $\delta_{\boldsymbol{\xi}}$ to represent the Dirac distribution supported on $\boldsymbol{\xi}$. In addition, we reserve the symbol ``$\;\widehat{}\;$'' for objects which are
  dependent  on the sample data. The symbols $\mathbb{E}$ and $\mathbb{P}$ denote, respectively, ``expectation'' and ``probability.'' Finally, for the rest of the paper we  assume that  we always have measurability for those objects, whose expected values we consider.

%



\section{Data-driven distributionally robust
optimization model}\label{ddro_section}

First, we briefly introduce some concepts from the \emph{optimal transport problem} (also known as the \emph{mass transportation problem}) that are at the core of  the development of our DRO framework.

Intuitively speaking, the optimal transport problem (OTP) centers on the question of how to
move masses between two probability distributions in such a way that the transportation
cost is minimal. Let $P$ and  $Q$  be two probability distributions in a
Polish space $S$ such that $P$
is the distribution of mass seen as the  origin
(i.e. the source) and $Q$ is the
distribution of mass seen as the
destination (i.e., the sink), and let $c$ be a measurable cost
function with $c(x,y)$ representing the cost
of moving a unit of mass from
location $x$ to location $y$. The OTP
can be stated as follows
 \begin{align*}
    C (P, Q)=\inf_{\Pi} \Big{\{ }   \int
    c(x,y) \Pi (dx, dy) :
     &\Pi \; \text{is a joint distribution}
       \;  \text{with marginals}
      \;P\;\text{and }\;  Q,\;  \\
      &\text{
      respectively} \Big{\} }   \nonumber
\end{align*}
We assume that the cost function $c$ is a
non-negative jointly convex lower
semicontinuous function such that if $x=y$  , then $c(x,y)=0 $. In the remainder of the paper we assume that we have existence and uniqueness of the OTP (see, for example,  Theorem 4.1 in \cite{Villani2008}).

For more technical
details about the assumptions
on the cost function in the OTP, we refer to
\cite{Villani2008} and
\cite{Santambrogio2015}. Note that if we choose a distance on $S$ as the cost function (for example, a $p$-norm, with $p \geq 1$, if $S$ is the Euclidean space $\mathbb{R}^{n}$), we get the so-called \emph{Wasserstein} metric of order 1, which we represent as $\mathcal{W}(P,Q)$ and which is also known as the Kantorovich metric.

It is
well known that this probability distance
metrizes the weak convergence property. Furthermore,
convergence with respect to the Wasserstein
metric of order 1 is
equivalent to weak convergence plus
convergence of the first  moment.
Wherever the Wasserstein metric of order 1 is used in this paper, we implicitly consider the set of all probability distributions with finite moment of order $1$. Likewise, we refer to the
\emph{Wasserstein ball of radius $r\geqslant
0$ centered at a certain nominal probability
distribution $P_0$ }, which we denote by
$\mathbb{B}_r(P_0)$, as the set of all
probability distributions whose Wasserstein
metric of order 1 to $P_0$ is
at most $r$.

\subsection{Formulation of the proposed model}
Problem (P) below formulates the data-driven distributionally robust optimization (DDRO) framework we propose.
\begin{subequations}
\begin{align}
\text{(P)}\;\;\inf_{\mathbf{x}\in X} \sup_{ Q \in \mathcal{Q}}& \  \mathbb{E}_{Q} \left[f(\mathbf{x},\boldsymbol{\xi}) \right] \label{OF_ini}\\
 \text{s.t.} \;&   \mathbb{P}_Q \left[ \boldsymbol{\xi} \in \Xi_i \right]=p_i, \forall i \in \mathcal{I}  \label{contraint_prob}\\
&  \widetilde{c}(\mathbf{p}-\widehat{\mathbf{p}})\leqslant \rho  \label{constraint_total_variation_p}  \\
& \sum_{i \in \mathcal{I}}p_i C(Q_i,\widehat{Q}_i)\leqslant \varepsilon\label{constraint_ball}\\
& Q_i \in \mathcal{Q}_i, \forall i \in \mathcal{I} \label{constraint_ball_2}\\
& \mathbf{p} \in \Theta \label{constraint_ball_3}
\end{align}
\end{subequations}
where   $X\subseteq \mathbb{R}^n$ is the set of feasible decisions, $\boldsymbol{\xi}: \Omega
\rightarrow \Xi \subseteq \mathbb{R}^d$   is a random vector defined on the
measurable space $(\Omega, \mathcal{F})$  with
$\sigma$-algebra $\mathcal{F}$, and $\mathcal{Q}$  is the set of all probability distributions over the
measurable space $(\Omega, \mathcal{F})$. Moreover,
for each $i \in \mathcal{I}$, $Q_i$ is the conditional distribution of $Q$ given
$\boldsymbol{\xi} \in \Xi_i$, that is $Q_i=Q(\boldsymbol{\xi}  \; /\; \boldsymbol{\xi} \in \Xi_i) \in
\mathcal{Q}_i$, with $\mathcal{Q}_i$ being the set of all conditional probability
distributions of $Q$ given $\boldsymbol{\xi} \in \Xi_i$.    In this setting, $\mathcal{I}$ is the
set of regions $\Xi_i$ with pairwise disjoint interiors into which the support set $\Xi$ is partitioned,  that is,
 $\bigcup_{i \in \mathcal{I}}(\Xi_i)= \Xi$ and $\text{int}(\Xi_i)\bigcap\text{int}(\Xi_j)=\emptyset
 $, $\forall i, j \in \mathcal{I}, i \neq j$.  Furthermore, we assume that $Q^*(\Xi_i \cap \Xi_j)=0 , \forall i, j \in \mathcal{I}, i \neq j$, where $Q^*$ is the true data-generating distribution. This is equivalent to stating that $\{  \Xi_i \}_{i \in \mathcal{I}}$ constitutes a
$Q^*$-packing (see a formal definition of this concept in page 50 of \cite{Graf2000}) and will allow us to unequivocally assign samples from $Q^*$ to the partitions $\Xi_i, i \in \mathcal{I}$. Finally, constraint~\eqref{constraint_total_variation_p}   defines the set of all probability vectors $\mathbf{p}$  that differ from the nominal empirical probability vector $\widehat{\mathbf{p}}$ in at most $\rho$ according to the cost function $\widetilde{c}$. This is a function that quantifies how  dissimilar two probability vectors $\mathbf{p}$ and $\mathbf{q}$ are. For this purpose, we require that $\widetilde{c}$ be  a
non-negative jointly convex lower
semicontinuous function such that if $\mathbf{p}=\mathbf{q}$, then $\widetilde{c}(\mathbf{p},\mathbf{q})=0 $. As mentioned further on, function $\widetilde{c}$ could, for example, take the form of a norm or a $\phi$-divergence.
 To ease the notation and the formulation, we  use
 $\boldsymbol{\xi}$
to represent either the random vector $\boldsymbol{\xi}(\omega)$,  with $ \omega \in \Omega$
 or an element of
$\mathbb{R}^d$. Note that we can consider the probability
measure induced by the random vector $\boldsymbol{\xi}$, if we choose
the  corresponding Borel $\sigma$-algebra $\mathcal{B}$ on $\Xi$. Thus, we can see $\mathcal{Q}$ as a  set of
probability measures defined over $(\Xi, \mathcal{B})$, so  we  write $\mathcal{Q}=\mathcal{Q}(\Xi)$. We define the uncertainty set $\mathcal{P}$ for the probability vector $\mathbf{p}
\in \mathbb{R}^{|\mathcal{I}|}$, with $|\mathcal{I}|$ being the number of partitions, as the intersection of $\Theta$ and the set defined by constraint~\eqref{constraint_total_variation_p}.
The support set $\Theta$, which includes the order cone constraints on the probability masses $\mathbf{p}$, is given by:
\begin{equation}\label{Theta_def}
    \Theta=\{ \mathbf{p}
\in \mathbb{R}^{|\mathcal{I}|} : \langle \mathbf{e},\mathbf{p}\rangle=1, \mathbf{p}\in \mathbb{R}^{|\mathcal{I}|}_{+}, \mathbf{p}\in \mathcal{C}  \}
\end{equation}
where $\mathcal{C}$ is a proper  (convex,
closed, full and pointed) cone. Hence, $\Theta $ is a  convex compact set. \\\\
In problem (P), $\rho$ and $\varepsilon$ are non-negative parameters, to be tuned by the decision maker, which control the size of the ambiguity set defined by equations \eqref{contraint_prob}--\eqref{constraint_ball_3}.

We represent this set as $\mathcal{U}_{\rho,\varepsilon} (\widehat{Q})$, where $\widehat{Q}$ is a nominal distribution expressed in terms of $\widehat{\mathbf{p}}$ and $\widehat{Q}_i$ as
\begin{equation}
 \widehat{Q}= \sum_{i \in \mathcal{I}}{\widehat{p}_i \widehat{Q}_i}
\end{equation}
where
\begin{equation}
 \widehat{p}_i=\frac{N_i}{N+|I'|} \label{nominal_vector_masses}
\end{equation}
and
\begin{equation}
 \widehat{Q}_i =\frac{1}{N_i}\sum_{j=1}^{N_i} \delta_{\boldsymbol{\widehat{\xi}}_j^{\,\, i}}
\end{equation}
Additionally, $I'=\{i \in \mathcal{I}\  \textrm{such that partition} \; i \textrm{ does not contain any data from the sample} \}$,
$\boldsymbol{\widehat{\xi}}_j^{\,\, i}\in \{\boldsymbol{\widehat{\xi}}_1^{\,\, i}, \ldots,
\boldsymbol{\widehat{\xi}}_{N_i}^{\,\, i} \}$ and $N_i$ is the number of atoms in region $\Xi_i$. Here we set $N_i=1$ and
$\boldsymbol{\widehat{\xi}}^{\,\, i}_1:=\arg \sup_{\boldsymbol{\xi}\in \Xi_i}f(\mathbf{x},\boldsymbol{\xi})$ for those $i \in I'$. Implicitly, we assume that this supremum is attained. We remark that this modeling choice protects the decision maker in those cases where there is a
total absence of information on the conditional distributions $Q_i, i \in I'$. Indeed, by introducing the ``artificial'' data point $\boldsymbol{\widehat{\xi}}^{\,\, i}_1:=\arg \sup_{\boldsymbol{\xi}\in \Xi_i}f(\mathbf{x},\boldsymbol{\xi})$ in a partition $\Xi_i$ with no samples, we are considering the worst-case form that the true conditional distribution $Q_i$ could possibly take, that is, a Dirac distribution supported on $\boldsymbol{\widehat{\xi}}^{\,\, i}_1$.

Finally, we note that the ambiguity set defined by constraints \eqref{contraint_prob}--\eqref{constraint_ball_3} is unequivocally determined by specifying the partitions $\Xi_i, i \in \mathcal{I}$, the nominal distribution $\widehat{Q}$, the budgets $\rho$ and $\varepsilon$, and the order cone constraints $\mathbf{p} \in \mathcal{C}$ in \eqref{Theta_def}. In fact, if these constraints are removed and we set $\rho=\varepsilon = 0$, then we have $p_i = \widehat{p}_i$ and $Q_i =\widehat{Q}_i, \forall i$, and therefore, $Q = \widehat{Q}$.

The following theorem shows that problem (P) can be reformulated as a single-level problem.
\begin{theorem}[Reformulation based on strong duality]
For any non-negative values of parameters $\varepsilon, \rho$, problem (P) is equivalent to the following:
\begin{align}
   \text{(P0)} \; \inf_{\mathbf{x},\lambda,\boldsymbol{\mu},\eta\;\widetilde{\mathbf{p}}, \theta, \mathbf{t} }& \lambda \rho+ \eta+\theta\varepsilon+\lambda \widetilde{c}^*_{\mathbf{\widehat{p}}}\left(\frac{\left(\frac{1}{N_i}\sum_{j=1}^{N_i}t_{i,j}\right)_{i \in \mathcal{I}}+\boldsymbol{\mu}-\eta
\mathbf{e}+\widetilde{\mathbf{p}}}{\lambda} \right)\notag\\
    \text{s.t.} \enskip & t_{i,j} \geqslant \sup_{\boldsymbol{\xi}
      \in \Xi_i}\left[f(\mathbf{x},\boldsymbol{\xi})-\theta
      c(\boldsymbol{\xi}, \widehat{\boldsymbol{\xi}}^{\, i}_j)\right],
      \;\forall i \in \mathcal{I}, \;j \leqslant N_i \label{P1_c}\\
      & \mathbf{x}\in X,\lambda \geqslant 0,\boldsymbol{\mu}\in \mathbb{R}^{|\mathcal{I}|}_{+},\eta \in  \mathbb{R},\;\widetilde{\mathbf{p}}\in
\mathcal{C}^*,\theta \geqslant  0 \notag\\
& t_{i,j} \in \mathbb{R}, \forall i \in \mathcal{I},  j \leqslant N_i \notag
\end{align}
where $\widetilde{c}^{*}_{\mathbf{\widehat{p}}}(\cdot)$ is the convex conjugate function of $\widetilde{c}(\cdot, \mathbf{\widehat{p}})$, with $\mathbf{\widehat{p}}$ fixed, and $\left(\frac{1}{N_i}\sum_{j=1}^{N_i}t_{i,j}\right)_{i \in \mathcal{I}}$ is the vector with the $|\mathcal{I}|$ components $\frac{1}{N_i}\sum_{j=1}^{N_i}t_{i,j}$.
\end{theorem}
\begin{proof}
Recall that we have assumed that regions $\Xi_i$ are disjoint. Thus, using the law of total probability, we can rewrite problem (P) as follows:
\begin{align}\label{problem_2}
\inf_{\mathbf{x}\in X} \sup_{\mathbf{p} \in \mathcal{P}}& G(\mathbf{x},\mathbf{p})
\end{align}
where we have considered the subproblem (SP):
\begin{subequations}
\begin{align}
  \text{(SP)}\; G(\mathbf{x},\mathbf{p})=& \sup_{Q_i \in \mathcal{Q}_i, \forall i}\sum_{i \in  \mathcal{I}} p_i \mathbb{E}_{Q_i}\left[f(\mathbf{x},\boldsymbol{\xi}) \right]\\
   \text{s.t.}    &\sum_{i \in \mathcal{I}}p_i C(Q_i,\widehat{Q}_i)\leqslant \varepsilon
\end{align}
\end{subequations}

The probability distribution $\widehat{Q}_i$ is defined as $\widehat{Q}_i=\frac{1}{N_i}\sum_{j=1}^{N_i} \delta_{\boldsymbol{\widehat{\xi}}_j^{\,\, i}}$, with
$\boldsymbol{\widehat{\xi}}_j^{\,\, i}\in \{\boldsymbol{\widehat{\xi}}_1^{\,\, i}, \ldots,
\boldsymbol{\widehat{\xi}}_{N_i}^{\,\, i} \}$, and $N_i$ being the number of data points in $\Xi_i$.

Note that the structure of problem~\eqref{problem_2} does not fit in the general ambiguity set proposed in~\cite{Chensim2020}.

Equivalently, we can recast the subproblem (SP) as
\begin{align}
    \text{(SP)} &= \left\{ \begin{array}{cl}\displaystyle \sup_{Q_i \in \mathcal{Q}_i, \Pi_i, \forall i}\sum_{i \in  \mathcal{I}} p_i &{}\displaystyle \int_{\Xi_i}f(\mathbf{x},
\boldsymbol{\xi}) Q_i (d\boldsymbol{\xi}) \\ \text {s.t.}&{}\displaystyle \sum_{i\in \mathcal{I}}p_i
\int_{\Xi_i^2}
    c(\boldsymbol{\xi}, \boldsymbol{\xi}') \Pi_i
    (d\boldsymbol{\xi}, d\boldsymbol{\xi}') \leqslant
    \varepsilon  \\[1ex] &{} \left\{ \begin{array}{l} \forall\; i,\; \Pi_i
    \text{
is } \text{ a } \text{ joint } \text{ distribution } \text{ of }
\boldsymbol{\xi} \text{ and } \boldsymbol{\xi}'\\ \text{ with }
\text{ marginals } Q_i \text{ and } \widehat{Q}_i\text{, } \text{
respectively } \end{array}\right. \end{array} \right. \\&=
\left\{ \begin{array}{cl} \displaystyle\sup_{\widetilde{Q}_j^i, \forall i \in \mathcal{I},j\leqslant N_i} &{}
\displaystyle \sum_{i \in \mathcal{I}}\frac{p_i}{N_i}\sum_{j=1}^{N_i}\displaystyle\int_{\Xi_i}f(\mathbf{x},
\boldsymbol{\xi}) \widetilde{Q}_j^i (d\boldsymbol{\xi}) \\ \text
{s.t.}&{}\displaystyle \sum_{i \in \mathcal{I}}\frac{p_i}{N_i}\sum_{j=1}^{N_i}\int_{\Xi_i}
    c(\boldsymbol{\xi}, \widehat{\boldsymbol{\xi}}^{\, i}_j)
    \widetilde{Q}_j^i (d\boldsymbol{\xi}) \leqslant \varepsilon
    \label{eps_constraint} \\
      &\displaystyle \int_{\Xi_i} \widetilde{Q}_j^i (d\boldsymbol{\xi})=1,\;
\forall i \in \mathcal{I},j \leqslant N_i \end{array} \right.
\end{align}
where reformulation \eqref{eps_constraint} follows on from the fact that the marginal
distribution of $\boldsymbol{\xi}'$ is the discrete uniform
distribution supported on points $\boldsymbol{\widehat{\xi}}^{\,\, i}_j$, $j = 1, \ldots, N_i$.
%
Thus, $\Pi_i$ is completely determined by the conditional
distributions  $\widetilde{Q}_j^i = \Pi_i(\xi,\xi'|\xi' = \widehat{\xi}_j^i)$, $\forall i \leqslant N_i$, that is,
$\Pi_i(d\boldsymbol{\xi}, d\boldsymbol{\xi}')=
\frac{1}{N_i}\sum_{j=1}^{N_i}
\delta_{\boldsymbol{\widehat{\xi}}_j^{\, i}}(d\boldsymbol{\xi}')
\widetilde{Q}_j^i(d\boldsymbol{\xi})$ \cite{MohajerinEsfahani2018}.

The mathematical program~\eqref{eps_constraint} constitutes a generalized moment problem over the normalized measures $\widetilde{Q}_j^i$, for which strong duality holds (see, for example, \cite{Shapiro2001}). We can, therefore, dualize the $\varepsilon$-budget constraint on the transport cost, thus obtaining:
%
\begin{align}
    \displaystyle \inf_{ \theta \geqslant 0}\sup_{\widetilde{Q}_j^i,\forall i \in \mathcal{I},j \leqslant N_i  }\ &
     \theta\varepsilon +  \sum_{i \in \mathcal{I}}
     \frac{p_i}{N_i}\sum_{j=1}^{N_i}\int_{\Xi_i}\left[f(\mathbf{x}
     ,\boldsymbol{\xi})-\theta c(\boldsymbol{\xi},
     \widehat{\boldsymbol{\xi}}^{\, i}_j)\right]\widetilde{Q}_j^i
     (d\boldsymbol{\xi})\\
      &   \text{s.t.} \ \int_{\Xi_i}
      \widetilde{Q}_j^i(d\boldsymbol{\xi})=1,\; \forall i \in \mathcal{I},j \leqslant N_i  \\
     = &\inf_{ \theta \geqslant 0}\  \theta\varepsilon + \sum_{i \in \mathcal{I}}
    \frac{p_i}{N_i}\sum_{j=1}^{N_i} \sup_{\widetilde{Q}_j^i}
     \int_{\Xi_i}\left[f(\mathbf{x},\boldsymbol{\xi})-
     \theta c(\boldsymbol{\xi},
     \widehat{\boldsymbol{\xi}}^{\, i}_j)\right]
     \widetilde{Q}_j^i(d\boldsymbol{\xi})\\
      & \hspace{35mm}   \text{s.t.} \; \int_{\Xi_i}
      \widetilde{Q}_j^i(d\boldsymbol{\xi})=1,\; \forall i \in \mathcal{I},j \leqslant N_i  \\
      =&\inf_{ \theta\geqslant 0}\  \theta\varepsilon +
      \sum_{i \in \mathcal{I}}
    \frac{p_i}{N_i}\sum_{j=1}^{N_i} \sup_{\boldsymbol{\xi} \in
      \Xi_i}\left[f(\mathbf{x},\boldsymbol{\xi})-\theta
      c(\boldsymbol{\xi},
      \widehat{\boldsymbol{\xi}}^{\, i}_j)\right]\\
    =&\inf_{ \theta, t_{ij}, \forall i \in \mathcal{I},j \leqslant N_i  }
    \theta\varepsilon +
    \sum_{i \in \mathcal{I}}
    \frac{p_i}{N_i}\sum_{j=1}^{N_i}t_{i,j} \\
    \label{problem_dualized11}
      & \hspace{12mm}  \text{s.t.} \; t_{i,j} \geqslant \sup_{\boldsymbol{\xi}
      \in \Xi_i}\left[f(\mathbf{x},\boldsymbol{\xi})-\theta
      c(\boldsymbol{\xi}, \widehat{\boldsymbol{\xi}}^{\, i}_j)\right],
      \;\forall i \in \mathcal{I}, \;j \leqslant N_i \\ \label{constraints_dualized_t}
     &\hspace{18mm}  \theta \geqslant 0
\end{align}
where the second equality derives from the fact that we can choose a Dirac distribution supported on $\Xi_i$ as $\widetilde{Q}_j^i$ .


Now,   dualizing the $\rho$-budget constraint on the transport cost in the inner supremum of problem \eqref{problem_2}, we obtain:
\begin{align}
\inf_{ \lambda \geqslant 0} & \lambda \rho +   \sup_{\mathbf{p} \in \Theta}\left[G(\mathbf{x},\mathbf{p})-\lambda  \widetilde{c}(\mathbf{p}, \mathbf{\widehat{p}})\right]\label{problem_dualized1}
\end{align}
Thus,
\begin{align}
\inf_{ \lambda \geqslant 0} & \lambda \rho +   \sup_{\mathbf{p} \in \Theta}\left[G(\mathbf{x},\mathbf{p})-\lambda  \widetilde{c}(\mathbf{p}, \mathbf{\widehat{p}})\right]\\
= \inf_{ \lambda \geqslant 0} & \lambda \rho +   \sup_{\mathbf{p} \in \Theta}\left[ \inf_{
\theta \geqslant 0,(t_{i,j})\;s.t. \eqref{problem_dualized11}}
    \theta\varepsilon +
    \sum_{i \in \mathcal{I}}
    \frac{p_i}{N_i}\sum_{j=1}^{N_i}t_{i,j}-\lambda  \widetilde{c}(\mathbf{p}, \mathbf{\widehat{p}})\right]
\end{align}
Since function $ \theta\varepsilon + \sum_{i \in \mathcal{I}}
    \frac{p_i}{N_i}\sum_{j=1}^{N_i}t_{i,j}-\lambda  \widetilde{c}(\mathbf{p}, \mathbf{\widehat{p}})$  is upper semicontinuous and  concave in $\mathbf{p}$ on the compact convex set $\Theta$  (recall that $\widetilde{c}$ is nonegative, lower semicontinuous, and convex in $\mathbf{p}$), and linear in $\theta$ and $t_{i,j}$ on the convex set defined by $\theta \geqslant 0$ and  \eqref{problem_dualized11}, we can apply Sion's min-max theorem (\cite{Sion1958}) and in this way, interchange the innest infimum with the outer supremum. Then, by merging the two infima, we arrive at
\begin{align*}
\inf_{ \lambda\geqslant 0,\theta \geqslant 0,(t_{i,j})} & \lambda \rho +  \theta\varepsilon + \sup_{\mathbf{p} \in \Theta}\left[
    \sum_{i \in \mathcal{I}}
    \frac{p_i}{N_i}\sum_{j=1}^{N_i}t_{i,j}-\lambda  \widetilde{c}(\mathbf{p}, \mathbf{\widehat{p}})\right]\\
 &   \text{s.t.} \; t_{i,j} \geqslant \sup_{\boldsymbol{\xi}
      \in \Xi_i}\left[f(\mathbf{x},\boldsymbol{\xi})-\theta_i
      c(\boldsymbol{\xi}, \widehat{\boldsymbol{\xi}}^i_j)\right],
      \;\forall i \in \mathcal{I}, \;j \leqslant N_i
\end{align*}

We focus now on the inner supremum,
\begin{equation}\label{conic_problem1}
 \sup_{\mathbf{p} \in \Theta}\left[ \left\langle \mathbf{p}, \left(\frac{1}{N_i}\sum_{j=1}^{N_i}t_{i,j}\right)_{i \in \mathcal{I}}\right \rangle-\lambda  \widetilde{c}(\mathbf{p}, \mathbf{\widehat{p}})\right]
\end{equation}
where we have written $ \sum_{i \in \mathcal{I}}
    \frac{p_i}{N_i}\sum_{j=1}^{N_i}t_{i,j}$ as $\Big\langle \mathbf{p},
    \left(\frac{1}{N_i}\sum_{j=1}^{N_i}t_{i,j}\right)_{i\in \mathcal{I}} \Big\rangle$. This is a concave maximization problem (be aware that $\langle \mathbf{p}, \mathbf{H}(\mathbf{x})\rangle-\lambda
\widetilde{c}(\mathbf{p}, \mathbf{\widehat{p}})$ is a concave function with respect to $\mathbf{p}$ and $\Theta$ is a convex
compact set; furthermore, notice that we have $\mathbf{H}(\mathbf{x}) = \left(\frac{1}{N_i}\sum_{j=1}^{N_i}t_{i,j}\right)_{i \in \mathcal{I}}$ in our particular case). Consequently, strong duality holds if a Slater condition is satisfied, that is, if there exists a point $\mathbf{p}^*
\in \text{relint}(\mathbb{R}^{|\mathcal{I}|}_+)$ such that $\langle \mathbf{e}, \mathbf{p}^* \rangle=1$,  and $\mathbf{p}^* \in
\text{int}(\mathcal{C})$ (see, for example, \cite{Boyd2004}).
Using a standard duality argument, we dualize the constraints $\mathbf{p}\in \mathbb{R}^{|\mathcal{I}|}_{+}$, $\langle
\mathbf{e},\mathbf{p}\rangle=1$ and $ \mathbf{p}\in \mathcal{C}$, with associated
multipliers $\boldsymbol{\mu}\in \mathbb{R}^{|\mathcal{I}|}_{+}$, $\eta \in \mathbb{R}$ and  $\widetilde{\mathbf{p}} \in
\mathcal{C}^*$,
respectively. Thus, we obtain the following problem:
\begin{align*}
&\inf_{\eta \in  \mathbb{R},\boldsymbol{\mu}\in \mathbb{R}^{|\mathcal{I}|}_{+},\;\widetilde{\mathbf{p}}\in
\mathcal{C}^* } \sup_{\mathbf{p}} \left\{ \left\langle \mathbf{p}, \left(\frac{1}{N_i}\sum_{j=1}^{N_i}t_{i,j}\right)_{i \in \mathcal{I}}\right\rangle-\lambda
\widetilde{c}(\mathbf{p}, \mathbf{\widehat{p}})+\langle \boldsymbol{\mu},\mathbf{p}\rangle+\eta (1-\langle
\mathbf{e},\mathbf{p}\rangle)+\langle \widetilde{\mathbf{p}},\mathbf{p}\rangle\right\}=\\
&\inf_{\eta \in  \mathbb{R},\;\boldsymbol{\mu}\in \mathbb{R}^{|\mathcal{I}|}_{+},\widetilde{\mathbf{p}}\in
\mathcal{C}^* } \eta+\sup_{\mathbf{p}} \left\{ \left \langle \mathbf{p}, \left(\frac{1}{N_i}\sum_{j=1}^{N_i}t_{i,j}\right)_{i \in \mathcal{I}}+\boldsymbol{\mu}-\eta
\mathbf{e}+\widetilde{\mathbf{p}}\right\rangle-\lambda  \widetilde{c}(\mathbf{p}, \mathbf{\widehat{p}})
\right\}=\\
&\inf_{\eta \in  \mathbb{R},\;\boldsymbol{\mu}\in \mathbb{R}^{|\mathcal{I}|}_{+},\widetilde{\mathbf{p}}\in
\mathcal{C}^* } \eta+\lambda\sup_{\mathbf{p}} \left\{ \left\langle \mathbf{p}, \frac{\left(\frac{1}{N_i}\sum_{j=1}^{N_i}t_{i,j}\right)_{i \in \mathcal{I}}+\boldsymbol{\mu}-\eta
\mathbf{e}+\widetilde{\mathbf{p}}}{\lambda}\right\rangle-  \widetilde{c}(\mathbf{p}, \mathbf{\widehat{p}})
\right\}=\\
&\inf_{\eta \in  \mathbb{R},\;\boldsymbol{\mu}\in \mathbb{R}^{|\mathcal{I}|}_{+},\widetilde{\mathbf{p}}\in
\mathcal{C}^* } \eta+\lambda \widetilde{c}^*_{\mathbf{\widehat{p}}}\left(\frac{\left(\frac{1}{N_i}\sum_{j=1}^{N_i}t_{i,j}\right)_{i \in \mathcal{I}}+\boldsymbol{\mu}-\eta
\mathbf{e}+\widetilde{\mathbf{p}}}{\lambda} \right)
\end{align*}
where $\widetilde{c}^{*}_{\mathbf{\widehat{p}}}(\cdot)$ is the convex conjugate function of $\widetilde{c}(\cdot, \mathbf{\widehat{p}})$, with $\mathbf{\widehat{p}}$ fixed.

Therefore, problem \eqref{problem_2} can be equivalently reformulated as follows:
\begin{align}
   \text{(P0)} \; \inf_{\mathbf{x},\lambda,\boldsymbol{\mu},\eta\;\widetilde{\mathbf{p}}, \theta, \mathbf{t} }& \lambda \rho+ \eta+\theta\varepsilon+\lambda \widetilde{c}^*_{\mathbf{\widehat{p}}}\left(\frac{\left(\frac{1}{N_i}\sum_{j=1}^{N_i}t_{i,j}\right)_{i \in \mathcal{I}}+\boldsymbol{\mu}-\eta
\mathbf{e}+\widetilde{\mathbf{p}}}{\lambda} \right)\notag\\
    \text{s.t.} \enskip & t_{i,j} \geqslant \sup_{\boldsymbol{\xi}
      \in \Xi_i}\left[f(\mathbf{x},\boldsymbol{\xi})-\theta
      c(\boldsymbol{\xi}, \widehat{\boldsymbol{\xi}}^{\, i}_j)\right],
      \;\forall i \in \mathcal{I}, \;j \leqslant N_i\notag\\ 
      & \mathbf{x}\in X,\lambda \geqslant 0,\boldsymbol{\mu}\in \mathbb{R}^{|\mathcal{I}|}_{+},\eta \in  \mathbb{R},\;\widetilde{\mathbf{p}}\in
\mathcal{C}^*,\theta \geqslant  0 \notag\\
& t_{i,j} \in \mathbb{R}, \forall i \in \mathcal{I},  j \leqslant N_i \notag
\end{align}

\end{proof}

Moreover, in the case that the cost function $\widetilde{c}(\cdot, \cdot)$ is given by a norm, we have $\widetilde{c}_{\mathbf{\widehat{p}}}(\mathbf{p})=\left\| \mathbf{p}-\widehat{\mathbf{p}}\right\|$. The next corollary deals with this particular case.
\begin{corollary}
If the cost functions $c(\cdot, \cdot)$ and $\widetilde{c}(\cdot, \cdot) $ are  given by  norms, then for any non-negative values of parameters $\varepsilon, \rho$, the problem (P) is equivalent to the following problem
 \begin{align}
 \text{(P1)}\;\;    \inf_{\mathbf{x},\lambda,\boldsymbol{\mu},\eta\;\widetilde{\mathbf{p}}, \theta, \mathbf{t}}& \lambda \rho+ \eta+\theta\varepsilon + \sum_{i \in \mathcal{I}} \widehat{p}_i\left(    \frac{1}{N_i}\sum_{j=1}^{N_i}t_{i,j}+\mu_i-\eta+\widetilde{p}_i\right)\notag\\
  \text{s.t.} \;     & t_{i,j} \geqslant \sup_{\boldsymbol{\xi} \in \Xi_i}\left[f(\mathbf{x},\boldsymbol{\xi})-\theta \left\|\boldsymbol{\xi}- \widehat{\boldsymbol{\xi}}^{\, i}_j\right\|\right], \forall i\in \mathcal{I}, \forall j\leqslant N_i \label{cons_sup_norm_1} \\
  & \left\|\left(\frac{1}{N_i}\sum_{j=1}^{N_i}t_{i,j}+\mu_i-\eta+\widetilde{p}_i\right)_{i \in \mathcal{I}}\right\|_{*} \leqslant \lambda \notag\\
  & \mathbf{x}\in X,\lambda \geqslant 0,\boldsymbol{\mu}\in \mathbb{R}^{|\mathcal{I}|}_{+},\eta \in  \mathbb{R},\;\widetilde{\mathbf{p}}\in
\mathcal{C}^*,\theta \geqslant  0 \notag\\
& t_{i,j} \in \mathbb{R}, \forall i \in \mathcal{I}, \forall j \leqslant N_i\notag
\end{align}
\end{corollary}
\begin{proof}

 We  use the following Lemma to put problem (P0) in a better shape.
 \begin{lemma}\label{conjnorm}
 Let $ \widetilde{c}_{\mathbf{\widehat{p}}}(\mathbf{p})=\left\| \mathbf{p}-\widehat{\mathbf{p}}\right\|$, where $\widehat{\mathbf{p}}\in \mathbb{R}^{|\mathcal{I}|}$ is a  fixed vector and $\|\cdot \|$ a norm in $\mathbb{R}^{|\mathcal{I}|}$. Then, it holds that the convex conjugate function of $\widetilde{c}_{\mathbf{\widehat{p}}}(\mathbf{p})$ is as follows
$$\widetilde{c}^{*}_{\mathbf{\widehat{p}}}(\mathbf{s})= \left\{
\begin{array}{lr}
\sum_{i \in \mathcal{I}}\widehat{p}_i s_i     &\text{if}\quad  \left\| \mathbf{s}\right\|_{*} \leqslant 1 \\
 \infty    &\text{if}\quad \left\| \mathbf{s}\right\|_{*} > 1
\end{array}\right.$$
 \end{lemma}
\begin{proof}
 The claim of the Lemma follows from  Proposition 5.1.4. (vii) and Example 5.1.2 (b) of \cite{Lucchetti2006}.\qed
\end{proof}

Therefore, problem (P0) reduces to
    \begin{align}
 \text{(P1)}\;\;    \inf_{\mathbf{x},\lambda,\boldsymbol{\mu},\eta\;\widetilde{\mathbf{p}}, \theta, \mathbf{t}}& \lambda \rho+ \eta+\theta\varepsilon + \sum_{i \in \mathcal{I}} \widehat{p}_i\left(    \frac{1}{N_i}\sum_{j=1}^{N_i}t_{i,j}+\mu_i-\eta+\widetilde{p}_i\right)\notag\\
  \text{s.t.} \;     & t_{i,j} \geqslant \sup_{\boldsymbol{\xi} \in \Xi_i}\left[f(\mathbf{x},\boldsymbol{\xi})-\theta \left\|\boldsymbol{\xi}- \widehat{\boldsymbol{\xi}}^{\, i}_j\right\|\right], \forall i\in \mathcal{I}, \forall j\leqslant N_i \notag \\
  & \left\|\left(\frac{1}{N_i}\sum_{j=1}^{N_i}t_{i,j}+\mu_i-\eta+\widetilde{p}_i\right)_{i \in \mathcal{I}}\right\|_{*} \leqslant \lambda \notag\\
  & \mathbf{x}\in X,\lambda \geqslant 0,\boldsymbol{\mu}\in \mathbb{R}^{|\mathcal{I}|}_{+},\eta \in  \mathbb{R},\;\widetilde{\mathbf{p}}\in
\mathcal{C}^*,\theta \geqslant  0 \notag\\
& t_{i,j} \in \mathbb{R}, \forall i \in \mathcal{I}, \forall j \leqslant N_i\notag
\end{align}
\qed
\end{proof}
 \textbf{Remarks.} Our data-driven DRO framework (P) can be easily understood as a generalization of other popular DRO approaches. To see this, first we need to remove the order cone constraints on the probabilities associated with each subregion into which the support $\Xi$ has been partitioned, that is, the condition $\mathbf{p} \in \mathcal{C}$, and then proceed as indicated below:

\begin{enumerate}
    \item If we set $\varepsilon=0$, $|\mathcal{I}|=N$, with every partition containing a single and different data point from the sample, and use a $\phi$-divergence to build the  cost function, i.e., $\widetilde{c}_{\mathbf{\widehat{p}}}(\mathbf{p})=\sum_{i \in \mathcal{I}}\widehat{p}_i \phi\left(\frac{p_i}{\widehat{p}_i}\right)$ and hence, $c^{*}_{\mathbf{\widehat{p}}}(\mathbf{s})=\sum_{i \in \mathcal{I}}\widehat{p}_i \phi^*(s_i)$, then our data-driven DRO approach boils down to that of   \cite{Ben-Tal2013} and \cite{Bayraksan2015}.
   \item On the contrary, if we set $|\mathcal{I}|=1$,
   $c$ is given by a norm
   and take  $\widetilde{c}_{\mathbf{\widehat{p}}}(\mathbf{p})=\left\| \mathbf{p}-\widehat{\mathbf{p}}\right\|$ (hence  $\widetilde{c}^{*}_{\mathbf{\widehat{p}}}(\mathbf{s})=\sum_{i \in \mathcal{I}}\widehat{p}_i s_i$ if $\left\| \mathbf{s}\right\|_{*} \leqslant 1$),
   we get the model of \cite{MohajerinEsfahani2018}.
\end{enumerate}

Finally, we remark that constraint \eqref{P1_c} for each $i \in I'$ is equivalent (under the assumptions we make on the transportation cost function) to $t_{i,1} \geqslant \sup_{\xi \in \Xi_i}f(x,\xi)$.

\subsection{Tractable reformulations}
In this section we provide \emph{nice} reformulations of our DRO model (P) under
mild assumptions. For this purpose, we make use of the theoretical foundations laid out in \cite{MohajerinEsfahani2018}. Likewise, some extensions to our model, such as the extension to two-stage stochastic programming problems, are omitted here for brevity and because they can be easily derived in a similar way as found  in \cite{MohajerinEsfahani2018} for the data-driven DRO approach they develop.

We start our theoretical development with the following assumption.

\begin{assumption}\label{assumption_convex}
We consider that $\Xi_i$, for each $i \in \mathcal{I}$, is a closed convex set, and that
$f(\mathbf{x}, \boldsymbol{\xi}):=$ $\max_{k \leqslant K} g_k(\mathbf{x},
\boldsymbol{\xi})$, with $g_k$, for each $k\leqslant K$, being a proper, concave
and upper
semicontinuous function with respect to $\boldsymbol{\xi}$ (for any fixed value of
$\mathbf{x}\in X$)  and not identically $\infty$ on
$\Xi_i$.
\end{assumption}

Theorem~\ref{T1} below provides a tractable reformulation of problem (P1) as a finite convex problem. For ease of notation,  we suppress the dependence on the variable $\mathbf{x}$ (bearing in mind that this dependence occurs through functions $g_k$, $k \leqslant K$).

\begin{theorem}\label{T1}
If Assumption \ref{assumption_convex} holds and if we choose a norm (in $\mathbb{R}^d$) as the transportation cost function $c$, then for any values of $\rho$ and $\varepsilon$, problem (P1) is equivalent to the following finite convex problem:
    \begin{align*}
 \text{(P1')}\;\;    \inf_{\mathbf{x},\lambda,\eta, \boldsymbol{\mu},\;\widetilde{\mathbf{p}},\mathbf{z}_{ijk},\mathbf{v}_{ijk}, \theta, \mathbf{t}}& \lambda \rho+ \eta+\theta\varepsilon + \sum_{i \in \mathcal{I}} \widehat{p}_i\left(    \frac{1}{N_i}\sum_{j=1}^{N_i}t_{i,j}+\mu_i-\eta+\widetilde{p}_i\right)\\ \nonumber
  \text{s.t.} \;     & t_{i,j} \geqslant [-g_k]^*(\mathbf{z}_{ijk}-\mathbf{v}_{ijk})+S_{\Xi_i}(\mathbf{v}_{ijk})-\langle \mathbf{z}_{ijk}, \widehat{\boldsymbol{\xi}}^{\, i}_j \rangle\\ \nonumber
  & \forall i\in \mathcal{I}, \forall j\leqslant N_i, \forall k \leqslant K\\
  &  \left\| \mathbf{z}_{ijk}  \right\|_{*} \leqslant \theta , \forall i\in
  \mathcal{I}, \forall j\leqslant N_i, \forall k \leqslant K  \\
  & \left\|\left(\frac{1}{N_i}\sum_{j=1}^{N_i}t_{i,j}+\mu_i-\eta+\widetilde{p}_i\right)_{i \in \mathcal{I}}\right\|_{*} \leqslant \lambda \\
  & \mathbf{x}\in X,\lambda \geqslant 0,\theta \geqslant  0,\eta \in  \mathbb{R}, \boldsymbol{\mu} \in \mathbb{R}_{+}^{|\mathcal{I}|},\;\widetilde{\mathbf{p}}\in
\mathcal{C}^*,\\
&  \mathbf{z}_{ijk},\mathbf{v}_{ijk} \in \mathbb{R}^d , \forall i \in \mathcal{I}, \forall j \leqslant N_i, \forall k \leqslant K\\
& t_{i,j} \in \mathbb{R}, \forall i \in \mathcal{I}, \forall j \leqslant N_i
\end{align*}

\end{theorem}
where $[-g_k]^*(\mathbf{z}_{ijk}-\mathbf{v}_{ijk})$ is the conjugate function of $-g_k$ evaluated at $\mathbf{z}_{ijk}-\mathbf{v}_{ijk}$ and $S_{\Xi_i}$ is the support function of $\Xi_i$.

\begin{proof}

In essence, the complexity of problem (P1) depends on our ability to reformulate the supremum in constraint~\eqref{cons_sup_norm_1} in a tractable manner. This is possible under Asummption~\ref{assumption_convex}, following similar steps to those in the proof of Theorem 4.2 in \cite{MohajerinEsfahani2018}, to which we refer. \qed
\end{proof}


We note that Asummption~\ref{assumption_convex} covers the particular case where functions $g_k$, $k \leqslant K$, are affine and, as a result, $f$ is convex piecewise linear. The single-item newsvendor problem, which we illustrate in the first part of Section~\ref{numerics}, constitutes a popular example of this case.


\subsubsection{Separable objective function}
Now we extend the results presented above to a class of objective  functions which are  additively separable with respect to the dimension $d$. We assume here that
$\boldsymbol{\xi}=(\boldsymbol{\xi}_1, \ldots, \boldsymbol{\xi}_d)$, where $\boldsymbol{\xi}_l \in \mathbb{R}^p$, for each $l=1,\ldots, d$.  Furthermore,  we consider the separable norm $\left\|\boldsymbol{\xi}\right\|_d:=\sum_{l=1}^d \left\| \boldsymbol{\xi}_l \right\| $
associated with the base norm $\left\| \cdot \right\|$ (on $\mathbb{R}^p$). Finally, we assume that the function $f$ is given as follows:
\begin{equation}
    f(\mathbf{x}, \boldsymbol{\xi})=\sum_{l=1}^d \max_{k \leqslant K}g_{lk} (\mathbf{x}, \boldsymbol{\xi}_l)
\end{equation}

In this case, the complexity of the resulting DRO problem is linear with respect to the number $N$ of samples. The multi-item newsvendor problem, which we illustrate in the second half of Section~\ref{numerics}, constitutes a popular example of this case.
\begin{theorem}
If $ f(\mathbf{x}, \boldsymbol{\xi})=\sum_{l=1}^d \max_{k \leqslant K}g_{lk} (\mathbf{x}, \boldsymbol{\xi}_l)$,  $\{g_{lk}\}_{k \leqslant K}$ satisfy Assumption \ref{assumption_convex} for all $l\leqslant d$, and  $\Xi_i$, for each $i \in \mathcal{I}$, is given by the Cartesian product of closed convex sets (that is,  $\Xi_i:=\prod_{l=1}^d D_l^i$, with $D_l^i$ a closed convex set), and if we choose the norm $\left\| \cdot \right\|_d$ as the transportation cost function $c$, then for any values of $\rho$ and $\varepsilon$,  problem (P)  is equivalent  to the following finite convex problem:
    \begin{align}
 \text{(P2)}\;\;    \inf_{\mathbf{x},\lambda,\eta,\boldsymbol{\mu},\;\widetilde{\mathbf{p}},\mathbf{z}_{ijkl}, \mathbf{v}_{ijkl}, \theta,\boldsymbol{\omega}}&  \lambda \rho+ \eta+\theta\varepsilon + \sum_{i \in \mathcal{I}} \widehat{p}_i\left(    \frac{1}{N_i}\sum_{j=1}^{N_i} \sum_{l=1}^d\omega_{ijl}+\mu_i-\eta+\widetilde{p}_i\right)\\
  \text{s.t.} \;     &  \omega_{ijl}\geqslant [-g_{lk}]^*(\mathbf{z}_{ijkl}-\mathbf{v}_{ijkl})+S_{D_l^i}(\mathbf{v}_{ijkl})-\langle
  \mathbf{z}_{ijkl}, \widehat{\boldsymbol{\xi}}^{\, i}_{jl} \rangle,\\ \nonumber
  & \forall i\in \mathcal{I}, \forall j\leqslant N_i, \forall k \leqslant K, \forall l\leqslant d \\
  &  \left\| \mathbf{z}_{ijkl} \right\|_{*} \leqslant \theta , \forall i\in \mathcal{I}, \forall j\leqslant N_i, \forall k \leqslant K, \forall l\leqslant d  \\
  & \left\| \left(    \frac{1}{N_i}\sum_{j=1}^{N_i}\sum_{l=1}^d\omega_{ijl}+\mu_i-\eta+\widetilde{p}_i\right) \right\|_{*} \leqslant \lambda \\
  & \mathbf{x}\in X,\lambda \geqslant 0,\theta \geqslant  0,\eta \in  \mathbb{R},\boldsymbol{\mu} \in \mathbb{R}_{+}^{|\mathcal{I}|},\;\widetilde{\mathbf{p}}\in
\mathcal{C}^*,\\
&  \omega_{ijl}\in \mathbb{R},  \forall i\in \mathcal{I}, \forall j\leqslant N_i,  \forall l\leqslant d\\
&  \mathbf{z}_{ijkl}, \mathbf{v}_{ijkl} \in \mathbb{R}^p, \forall i\in \mathcal{I}, \forall j\leqslant N_i, \forall k \leqslant K, \forall l\leqslant d
\end{align}

\end{theorem}
\begin{proof}The proof runs in a similar way to that of Theorem 6.1 in \cite{MohajerinEsfahani2018}.  \qed
\end{proof}
\textbf{Remarks.}
If the transportation cost function $c$ is not a norm, there are still some cases where the constraint
\eqref{P1_c} can be reformulated in a tractable way. In general, equation \eqref{P1_c} can be seen as the robust counterpart of a constraint affected by the random parameter vector $\boldsymbol{\xi}$, with $\Xi_i$ playing the role of the so-called \emph{uncertainty set}. In our case, the tractability of \eqref{P1_c} depends on the nature of each set $\Xi_i$ and each function $\alpha_{ij}(\boldsymbol{\xi}):= f(\mathbf{x},\boldsymbol{\xi})-\theta
      c(\boldsymbol{\xi}, \widehat{\boldsymbol{\xi}}^{\, i}_j)$. Indeed, suppose that every $\Xi_i$ is a closed convex set, then:

\begin{itemize}
    \item If the function $f$ is concave in $\boldsymbol{\xi}$, so is each function  $\alpha_{ij}(\boldsymbol{\xi})$ (recall that the transportation cost function $c$ is assumed to be convex and that $\theta$ is non-negative). As \cite{Roos2018a} points out, this is a tractable instance and tractable reformulations of constraint \eqref{P1_c}  can be obtained using Fenchel duality following the guidelines in \cite{Ben-Tal2015}.
    \item In contrast, the case in which some $\alpha_{ij}(\boldsymbol{\xi})$ are convex is much more challenging and may call for approximation methods such as the one proposed by \cite{Roos2018a}.
\end{itemize}

In any case, we need to compute convex conjugate functions, which is, in itself, a complicated problem in general. For assistance in this regard, one may resort to symbolic computation in order to get closed formulas for
convex conjugate functions (see, for example, \cite{Borwein2009}).

\subsection{Order cone constraints}\label{order_conic_section}
To account for \emph{a-priori} knowledge about the probability distribution of the random parameter vector $\boldsymbol{\xi}$ (for example, the decision maker may have some information about the shape of this distribution), we propose to convey this knowledge using order constraints on the probability masses $p_i$ associated with each subregion $\Xi_{i}$ into which the support $\Xi$ of $\boldsymbol{\xi}$ is partitioned. These order constraints are based on order cones, which, in turn, can be represented in the form of graphs.

We can build order cones from graphs that allow for the comparison of all probabilities $p_i$. In that case, we say that the graph, and the associated cone, establish a \emph{total order}. If, on the contrary, the graph only allows  some of those probabilities to be compared, we talk about \emph{partial order}. For more details about order cones we refer the reader to \cite{NEMETH201680}.

Below, we present  some common choices of order cones.


\begin{itemize}
    \item \emph{Simple order cone (monotonicity):}
    $$\mathcal{C}=\{ p \in \mathbb{R}^{|\mathcal{I}|} : p_1 \geqslant \ldots \geqslant p_{|\mathcal{I}|} \}$$
    \item \emph{Tree order cone:}
     $$\mathcal{C}=\{ p \in \mathbb{R}^{|\mathcal{I}|} : p_i \geqslant  p_{|\mathcal{I}|},\; i=1,\ldots, |\mathcal{I}|-1 \}$$
    \item \emph{Star-shaped cone (decrease on average):}
     $$\mathcal{C}=\Big\{ p \in \mathbb{R}^{|\mathcal{I}|} : p_1 \geqslant \frac{p_1+p_2}{2}\geqslant \ldots \geqslant \frac{p_1+\ldots+p_{|\mathcal{I}|}}{|\mathcal{I}|} \Big\}$$

    \item \emph{Umbrella cone (unimodality):}
     $$\mathcal{C}=\{ p \in \mathbb{R}^{|\mathcal{I}|} : p_1 \leqslant p_2 \leqslant \ldots \leqslant p_m \geqslant p_{m+1}  \geqslant \ldots \geqslant p_{|\mathcal{I}|}  \}$$
\end{itemize}

An order cone is a \emph{polyhedral convex cone} and as such, can be algebraically expressed in the form
$\mathcal{C} = \{\mathbf{p} \in \mathbb{R}^{|\mathcal{I}|}: \mathbf{A}\mathbf{p} \geqslant 0\}$,
with $\mathbf{A}$ being a matrix of appropriate dimensions. Its dual $\mathcal{C}^*$
can, therefore, be easily computed as $\mathcal{C}^{*} = \{\mathbf{\widetilde{p}} = \mathbf{A}^T
\boldsymbol{\nu}: \boldsymbol{\nu} \geqslant \mathbf{0}\}$ (see, for instance, Corollary 3.12.9 in \cite{Silvapulle2011}). Notwithstanding, our DRO approach can be equally applied under other types of support sets, as long as the problem \eqref{conic_problem1} admits a strong dual (we refer the interested reader to \cite{Ben-Tal2013} for a list of types of support sets under which strong duality holds).


As compared to other approaches available in the technical literature, order cones provide a straightforward way of encoding modality information in the ambiguity set of the DRO problem. For instance, \cite{Hanasusanto2015} indirectly introduces multi-modality information by imposing first and second moment conditions on the different ambiguous components of a mixture with known weights. Their approach, however, results in a semidefinite program. Unlike \cite{Hanasusanto2015}, the authors in \cite{Li2019mapr} explicitly incorporate modality information into their ambiguity set through moment and generalized unimodal constraints. Nonetheless, they still need to solve a semidefinite program and their DRO approach overlooks the data-driven nature of those constraints. In \cite{Chen2019}, they construct an ambiguity set made up of those absolutely continuous probability distributions whose density function is bounded by some bands with a certain confidence level. Their approach can be used to impose monotonicity or unimodality of the probability distributions, but can only be applied to the univariate case.

Beyond modality, the order cone constraints on the partition probabilities that characterize our DRO approach equip the decision maker with a versatile and intuitive framework to exploit information on the shape of the ambiguous probability distribution. For example, as we do in the numerical experiments in Section~\ref{numerics}, we can construct an order cone that constrains the ratios among the partition probabilities, which can be seen as a discrete approximation of encoding ``derivative'' information on the ambiguous probability distribution (if this admits a density function). Likewise, other order cones could be used to bestow some sense of ``convexity" on this distribution.





\section{On convergence and out-of-sample performance guarantees}\label{theory_section}

In this section, we show that our DRO approach (P) naturally inherits the convergence and performance guarantees of that introduced in \cite{MohajerinEsfahani2018}. For this purpose, we first need to recall some terminology and concepts
from this paper to which we will resort later on. Throughout this section,
we denote the
training data sample (that is, the sample path sequence) as   $\widehat{\Xi}_N:=\{
\widehat{\boldsymbol{\xi}^i} \}_{i=1}^N \subseteq \Xi$. Following
\cite{MohajerinEsfahani2018}, $\widehat{\Xi}_N$ can be seen
as a random vector governed by the probability distribution
$\mathbb{P}^N:=Q^* \times \cdots \times Q^*$ ($N$ times)
supported on $\Xi^N$ (with the respective product $\sigma$-algebra).

In the remainder of this paper, we will denote
 the optimization problem  associated with problem $(P)$
 under the true probability distribution $Q^*$ as (${\rm P}^*$) (that is,
 the problem defined as $J^*:=\inf_{\mathbf{x} \in X}
 \mathbb{E}_{Q^*} [f(\mathbf{x},
 \boldsymbol{\xi})]$).
We then say that a \emph{data-driven solution}
  for problem (${\rm P}^*$) is a feasible solution $\widehat{\mathbf{x}}_N \in X$ which is constructed
  from the sample data. Furthermore, the \emph{out-of-sample performance} of a data-driven solution $\widehat{\mathbf{x}}_N
  $ is defined as $\mathbb{E}_{Q^*} [f(\widehat{\mathbf{x}}_N ,\boldsymbol{\xi})]$.

%
In line with \cite{MohajerinEsfahani2018}, given a data-driven solution  $\widehat{\mathbf{x}}_N$,  a \emph{finite sample guarantee} is a relation in the form
 \begin{equation}\label{bounds_finite_sample_guarantee}
      \mathbb{P}^N\Big[ \widehat{\Xi}_N   \; : \;\mathbb{E}_{Q^*}[f(\widehat{\mathbf{x}}_N ,\boldsymbol{\xi})]\leqslant \widehat{J}_N\Big]
      \geqslant 1-\beta
  \end{equation}
  where $\widehat{J}_N$ is a \emph{certificate} for the out-of-sample
performance of $\widehat{\mathbf{x}}_N$ (i.e., an upper bound that is generally contingent on the training
dataset),   $\beta \in (0,1) $ is a \emph{significance parameter} with respect to the
distribution $\mathbb{P}^N$, on which both $\widehat{\mathbf{x}}_N$  and $\widehat{J}_N$ depend. Moreover, we refer to the probability on the left-hand side
of \eqref{bounds_finite_sample_guarantee} as  the \emph{reliability} of $(\widehat{\mathbf{x}}_N, \widehat{J}_N)$.

Ideally, we strive to develop a method capable of identifying a highly reliable data-driven solution with a certificate as low as possible.

The data-driven DRO approach that we propose in this paper to address problem (${\rm P}^*$) accounts for the uncertainty about the true data-generating distribution $Q^*$, while  taking advantage of some a-priori \emph{order} information that the decision maker may have on some probabilities induced by $Q^*$ over a partition of the support set $\Xi$.
Below, we claim  that the pair ($\widehat{\mathbf{x}}_N$, $\widehat{J}_N$) provided by our distributionally robust optimization problem (P) features performance guarantees in line with those discussed in \cite{MohajerinEsfahani2018}. More specifically, for a suitable choice of the ambiguity set, the optimal value $\widehat{J}_N$ of problem (P) constitutes a certificate of the type \eqref{bounds_finite_sample_guarantee} providing a confidence level $1-\beta$ on the out-of-sample performance of the data-driven solution $\widehat{\mathbf{x}}_N$. This can be formally stated under some assumptions about the underlying true conditional probability distributions.

To this end, we first provide probabilistic guarantees on the partition probabilities $p_i$, $\forall i \leqslant |\mathcal{I}|$. In this vein, note that the empirical probability $\widehat{p}_i$, defined as in Equation~\eqref{nominal_vector_masses}, can be modeled as a binomial distribution with success probability $p_i^*$, divided by the total number of trials. Consequently, by the Strong Law of Large Numbers (SLLN), $\widehat{p}_i$ converges to $p_i^*$ almost surely.

Now suppose that we choose a $\phi$-divergence as $\widetilde{c}$, where $\phi$ is a twice continuously differentiable function around 1 with $\phi''(1)>0$. Then, take $\beta_p>0$. If we choose  as $\rho$ the value
\begin{equation}\label{conf_set_divergence}
\rho(\beta_p):=(\phi''(1)/(2N))\chi^2_{|\mathcal{I}|-1,1-\beta_p}
\end{equation}
we get a confidence set of level $1-\beta_p$  on the true partition probabilities $\mathbf{p}^*$ (see \cite{Ben-Tal2013} and \cite{Bayraksan2015}).

If, alternatively, we choose the total variation distance as $\widetilde{c}$, we can use Equation (19) in \cite{Guo2019} to take  $\rho$ as
\begin{equation}\label{conf_set_total_var}
\rho(\beta_p):= (|\mathcal{I}|/\sqrt{N})(2+\sqrt{2\log(|\mathcal{I}|/\beta_p)})
\end{equation}
and obtain a confidence set of level $1-\beta_p$ on $\mathbf{p}^*$.

Next we establish  a concentration tail inequality of the probability weighted Wasserstein
metric of order 1 between each conditional distribution and its respective true conditional distribution. For this purpose, we first need  to make the following assumption:
\begin{assumption}[Light-tailed Conditional  Distributions]\label{assumption_light_tailed_cond}
For each $i \in \mathcal{I}$, there exist $a_i,\gamma_i \in \mathbb{R}$, with $a_i>1$ and $\gamma_i>0$ such that
\begin{align}
\mathbb {E}_{Q_i^*}\big [ \exp (\gamma_i \|\boldsymbol{\xi}\|  ^{a_i}) \big ]
= \int _{\Xi_i }
\exp (\gamma_i \|\boldsymbol{\xi}\|  ^{a_i})\,Q_i^*
(\mathrm {d}\boldsymbol{\xi }) < \infty .
\end{align}

\end{assumption}

The following theorem provides a tail concentration inequality for the weighted sum of the
Wasserstein
metrics of order 1 between the true and empirical conditional distributions.
\begin{theorem}[Concentration Inequality for the Conditional Distributions
]\label{th_conc_ineq_cond}
If Assumption \ref{assumption_light_tailed_cond} holds,  for each $i \in
\mathcal{I}$, given $\beta_i \in (0,1]$  we have that
$\forall N_i \geqslant 1$, $\dim (\boldsymbol{\xi}) \neq 2$ and for all $\varepsilon > \sum_{i \in \mathcal{I}}p_i\varepsilon _{N_i}(\beta_i ) $, for any values $p_i, i \in \mathcal{I}$ such that $p_i\geqslant0$ and $ \sum_{i\in \mathcal{I}}p_i=1$, the following holds
\begin{align}
\mathbb {P} \left[ \sum_{i \in \mathcal{I}} p_i \mathcal{W}(Q_i^*,\widehat{Q}_i\big ) \leqslant \varepsilon \right] \geqslant 1-\sum_{i \in \mathcal{I}}\beta_i \label{conc_ineq_cond_weight}
\end{align}
where
\begin{align} \varepsilon _{N_i}(\beta_i ) {:=}\left\{ \begin{array}{ll} \Big ({\log (B_i\beta_i ^{-1}) \over C_iN_i}
\Big )^{1/\max\{\dim (\boldsymbol{\xi}),2 \}} &{} \quad \text {if } N_i \ge {\log (B_i\beta_i ^{-1}) \over C_i}, \\ \Big ({\log (B_i\beta_i ^{-1}) \over C_iN_i} \Big )^{1/a_i} &{} \quad \text {if } N_i < {\log (B_i\beta_i ^{-1}) \over C_i}. \end{array}\right. \end{align}
\end{theorem}
\emph{Proof.} Given Assumption 3, for all $i \in \mathcal{I}$, we deduce from Theorem 2 in \cite{Fournier2015} that
$$\mathbb {P}  \left[  \mathcal{W}(Q_i^*,\widehat{Q}_i ) \leqslant
\varepsilon _{N_i}(\beta_i ) \right] \geqslant 1- \beta_i.$$
Thus, we have that

\begin{align}
\mathbb{P} \left[ \sum_{i \in \mathcal{I}}p_i \mathcal{W}(Q_i^*,\widehat{Q}_i )
\leqslant \sum_{i \in \mathcal{I}}p_i \varepsilon _{N_i}(\beta_i )
\right] \geqslant & \;
\mathbb{P} \left[ \bigcap_{i \in \mathcal{I}} \left( p_i\mathcal{W}(Q_i^*,\widehat{Q}_i )\leqslant p_i  \varepsilon _{N_i}(\beta_i ) \right) \right]\\
=& 1- \mathbb{P} \left[ \bigcup_{i \in \mathcal{I}} \left( p_i\mathcal{W}(Q_i^*,\widehat{Q}_i )> p_i  \varepsilon _{N_i}(\beta_i ) \right) \right]\\
&\geqslant 1- \sum_{i \in \mathcal{I}}\mathbb{P} \left[p_i\mathcal{W}(Q_i^*,\widehat{Q}_i )> p_i  \varepsilon _{N_i}(\beta_i ) \right]\\
&\geqslant 1-\sum_{i \in \mathcal{I}}\beta_i
\end{align}
Theorem~\ref{th_conc_ineq_cond} sets the probabilistic bound
$\sum_{i \in \mathcal{I}}p_i \varepsilon _{N_i}(\beta_i )$
on the weighted
Wasserstein metric of order 1 between each conditional distribution and its respective true conditional distribution, with  at least  confidence level $1-\sum_{i \in \mathcal{I}}\beta_i$. We remark that, if the partitions are compact, stronger results like those in Theorem 2 of \cite{JiLejeune2020} could be used to choose the radii of the Wasserstein balls. More specifically, the result in Theorem 2 of \cite{JiLejeune2020} depends on the diameter of the compact support set (i.e., the maximum distance between two elements of that set). The result stated in our theorem, in contrast, is valid for unbounded partitions, as it only requires the true conditional distribution associated with each partition be light-tailed.
The next theorem states the finite-sample guarantee performance of the proposed DRO method we develop in this paper:

\begin{theorem}[Finite sample guarantee]\label{finite_sample_theorem}
Suppose that Assumption \ref{assumption_light_tailed_cond}  holds and that we have chosen as $\rho$ the value
given by Equation~\eqref{conf_set_divergence} or \eqref{conf_set_total_var}.
Then, the finite sample guarantee \eqref{bounds_finite_sample_guarantee}
holds with at least confidence level $(1-\beta_p)(1-\sum_{i \in \mathcal{I}}\beta_i)$.
\end{theorem}
\emph{Proof.}
The claim follows from Theorem \ref{th_conc_ineq_cond} and Equations~\eqref{conf_set_divergence} and \eqref{conf_set_total_var}, which imply
that $\mathbb{P}(Q^* \in
\mathcal{U}_{\rho,\varepsilon} (\widehat{Q}))
\geqslant (1-\beta_p)(1-\sum_{i \in \mathcal{I}}\beta_i)$. Hence,
$$\mathbb{E}_{Q^*}[f(\widehat{\mathbf{x}}_N ,\boldsymbol{\xi})]\leqslant \sup_{Q \in
\mathcal{U}_{\rho,\varepsilon} (\widehat{Q}) }
\mathbb{E}_{Q}[f(\widehat{\mathbf{x}}_N ,\boldsymbol{\xi})]=\widehat{J}_N$$ with
probability at least $(1-\beta_p)(1-\sum_{i \in \mathcal{I}}\beta_i)$.

\textbf{Remarks.}
In practice, proper values for
$\varepsilon$ and $\rho$  are set by way of data-driven procedures   like bootstrapping or cross-validation, as we illustrate in the numerical experiments in Section 4.2 (see also \cite{JMLRRuidiChen}, \cite{cisneros20a}, \cite{MohajerinEsfahani2018}, \cite{JMLRShafieezadeh-Abadeh}, and \cite{Xie2021mapr} for more examples). These procedures allow the decision maker to tune those parameters as a function of  sample size $N$ in order to get reliable decisions without giving up too much on out-of-sample performance. Following this line, and as noted in Remark 5 in \cite{Kuhn2019},
    the requirement to include the true distribution inside the ambiguity set is only a sufficient, but \emph{not necessary} condition to ensure a finite sample guarantee. Indeed, this guarantee can be sustained even if the parameters of the ambiguity set  are reduced below the lowest values for which the ambiguity set represents a confidence set for the true distribution.

Furthermore, recall that the partition probabilities $\mathbf{p}$ belong to the support set $\Theta$ defined by the order cone constraints. Since we assume that these constraints are coherent with the true distribution $Q^*$, we do not need to explore those probability measures $Q$ in the Wasserstein ball $\mathbb{B}_{\rho_N (\beta)}$ that do not comply with them. Consider, for example, the case in which the worst-case
distribution in the ball $\mathbb{B}_{\rho_N (\beta)}$ does not satisfy the order cone constraints. One could expect, therefore, that, in practice, our approach could benefit from this fact to produce a data-driven solution $\widehat{\mathbf{x}}_N$ as reliable as that given by the method of \cite{MohajerinEsfahani2018}, but with a tighter certificate $\widehat{J}_N$. This is precisely what we observe in the numerical experiments that we present below.

We conclude this section with some remarks on the convergence and asymptotic consistency of our DRO approach: We have that, as the number $N$ of samples grows to infinity,
     \begin{equation}
        (\widehat{\mathbf{x}}_N,\widehat{J}_N) \rightarrow (\mathbf{x}^*,J^*)
     \end{equation}
     where $\mathbf{x}^*$ (resp. $J^*$)  is an optimizer (resp. the optimal solution value) of problem (${\rm P}^*$).

     Indeed, assume that Theorem 3.6 in \cite{MohajerinEsfahani2018} holds, then take a confidence level $1-\beta$, and choose $\varepsilon$ and $\rho$ by way of
     Theorem \ref{th_conc_ineq_cond} and Equations~\eqref{conf_set_divergence} (or \eqref{conf_set_total_var}), respectively. When $N$ grows to infinity, we have, on the one hand, that the  conditional distributions
     converge (in the Wasserstein metric) to their respective true conditional distributions and the probability weights converge a.s. by the SLLN to their respective true values.
     Therefore, both $\varepsilon$ and $\rho$ tend to zero as $N$ increases to infinity. Consequently,  our ambiguity set only contains the empirical distribution $\widehat{Q}_N$, which  converges almost surely to the true distribution $Q^*$.





\section{Numerical experiments}\label{numerics}
The following simulation experiments are designed to provide additional
insights into
the performance guarantees of our proposed distributionally robust
optimization scheme with order cone
constraints. For this purpose, we consider two test instances: the (single and multi-item) newsvendor problem and the problem of a strategic firm competing \`a la Cournot in a market.
 These two problems have been intentionally selected, because they are qualitatively different when addressed by the standard Wasserstein-metric-based DRO approach proposed in \cite{MohajerinEsfahani2018}. In effect, the former features an objective function $f(\mathbf{x},\boldsymbol{\xi})$ whose Lipschitz  constant with respect to $\boldsymbol{\xi}$ is independent of the decision $\mathbf{x}$. Consequently, as per Remark 6.7 in \cite{MohajerinEsfahani2018}, the standard Wasserstein-metric-based DRO approach renders the same minimizer for this problem as the sample average approximation, whenever the support of the uncertainty $\boldsymbol{\xi}$ is unbounded. This is, in contrast, not true for the problem of a strategic firm competing \`a la Cournot in a market, which is characterized by an objective function with a Lipschitz constant over $\boldsymbol{\xi}$ that is a function of $\mathbf{x}$.
 This  allows us to highlight the differences of our approach with regard to \cite{MohajerinEsfahani2018} in two distinct settings.

All the numerical experiments have  been implemented in Python. The optimization problems have been built using Pyomo \cite{Pyomo}  and solved with CPLEX 12.10 \cite{CPLEX}  on a PC with Windows 10 and a CPU Intel (R) Core i7-8550U clocking at 1.80 GHz and with 8 GB of RAM. The statistical methods that have been used for the numerical experiments have been coded by means of the module Scikit-learn (see  \cite{Pedregosa:2011:SML:1953048.2078195}). In what follows we provide  some implementation details regarding the proposed model.
The numerical experiments have been designed under the following assumptions:
\begin{enumerate}

    \item \emph{A-priori information}. Given a fixed and known partition of the sample space $\Xi$, we
    can construct an order cone that is consistent with the true probability
    distribution. That is,   the probability masses that the true distribution assigns to each partition
    verify the order cone constraints. In practice, this a-priori information is determined by the nature of the problem and the random phenomena, and is assumed to be known by the decision maker based on experience and expert knowledge.  Furthermore, in the case that the decision maker has no full certainty about the a-priori information, s/he may resort to statistical hypothesis testing to assess the confidence that the partition probabilities belong to a given order cone (see, for instance, \cite{Bhattacharya1997} and references therein).

    In our
    numerical experiments, we specifically apply the following approach: Given a fixed
    number of partitions (later we explain how the partition set is obtained), we consider that the
    decision maker knows a total order between the probability masses associated with each of the regions into which the sample space $\Xi$ is split. Furthermore, s/he also knows their ratios approximately, within a certain tolerance (which, in the subsequent experiments, we set to 0.1).

    For instance, suppose we have three partitions with (true) probability masses of $p_1^* = 0.6$, $p_2^* = 0.3$ and $p_3^* = 0.1$. The decision maker only knows their relative ratios with a tolerance error of 0.1, that is:
\begin{align*}
    p_1&\geqslant (0.6/0.3 -0.1)p_2\\
    p_2 & \geqslant (0.3/0.1-0.1)p_3
\end{align*}
   This way, we get the following  order cone constraints:
   \begin{align*}
    p_1&\geqslant 1.9p_2\\
    p_2 & \geqslant 2.9p_3
\end{align*}
    \item \emph{Support set $\Xi$}. The support set is the Cartesian product of
    closed intervals (that is, an hypercube, whose size is indicated in each example) and, therefore, is
    a closed convex set.

    \item \emph{True distribution}. For simulation and analysis, the data-generating distribution is approximated by a certain number of  data points (15\,000 in the newsvendor setting and 10\,000 in the problem of the Cournot producer) drawn from a mixture of three normal distributions, whose characteristics are specified in each of the two examples we consider in the following subsections. Furthermore, those data points that fall outside the support set $\Xi$ are discarded.

    \item \emph{Construction of partitions $\Xi_i$, $i = 1, \ldots, |\mathcal{I}|$:} In order to construct the partitions, we proceed as follows:
    \begin{enumerate}

    \item \emph{Clustering phase:} Firstly, we employ the $K$-means clustering technique to group the total number of data points that approximate the true data distribution into $K$ clusters. The number $K$ of clusters is decided upon using the well-known Elbow's method (see, for example, \cite{Dangeti:2017:SML:3164859}). It is based on the value of the average distortion produced by different
values of $K$. If $K$ increases, the average distortion will decrease and  the improvement in average distortion will diminish.  The
value of $K$ at which the improvement in distortion decreases the most is called the \emph{elbow}. At this value of $K$,
we should stop dividing the data into further clusters and choose this value as the number of clusters. In addition, we assign a label to identify each of the $K$ clusters. In all the numerical experiments that are presented next,  the true data-generating distribution is constructed as a mixture of three (univariate or multivariate) normal distributions. We assume that the decision maker has a good estimate of the number of components of this mixture and thus, we consider, for example, four clusters, i.e., $K = 4$.

    \item  \emph{Decision-tree classifier phase:} Once all the clusters have been labelled, we use the aforementioned total number of data points to train a decision-tree multi-classifier with a maximum number of leafs equal to $K$. The tree will be then used to allocate new data points into one of the $K$ clusters, which, in effect, is equivalent to having a partition of the support set in $K$ disjoint regions.

    \end{enumerate}

\item  \emph{Comparative analysis:} We compare
three different data-driven approaches to address the solution to problem $\inf_{\mathbf{x}\in X} \  \mathbb{E}_{Q^*} \left[f(\mathbf{x},\boldsymbol{\xi})\right]$,
namely, our approach (DROC), the one of \cite{MohajerinEsfahani2018} (DROW) and the sample average approximation (SAA). Recall that we denote $x^* \in arg\,min_{\mathbf{x}\in X} \  \mathbb{E}_{Q^*} \left[f(\mathbf{x},\boldsymbol{\xi})\right]$ and $J^* = \mathbb{E}_{Q^*} \left[f(\mathbf{x^*},\boldsymbol{\xi})\right]$, which, in practice, are unknown to the decision maker, but, for analysis purposes, we estimate
using the total number of data points that approximate the true data-generating distribution.
Moreover, in all numerical experiments, we consider the 1-norm as the functions $c$ and $\widetilde{c}$.
To compare the three data-driven approaches we consider, we use two performance metrics, specifically, the \emph{out-of-sample performance} of the data-driven solution (which we also refer to as its
\emph{actual expected cost}) and its \emph{out-of-sample disappointment}. The former is given by $\mathbb{E}_{Q^*} \left[f(\widehat{\mathbf{x}}_{N}^{m},\boldsymbol{\xi})\right]$, while the latter is calculated as $J^* -\widehat{J}_{N}^{m}$, where $m = \{\textrm{DROC,}$ $ \textrm{DROW, SAA}\}$ and $\widehat{J}_{N}^{m}$ is the objective function value yielded by the data-driven optimization problem solved by method $m$. We stress that a negative out-of-sample disappointment represents a favourable outcome.
As $\mathbb{E}_{Q^*} \left[f(\widehat{\mathbf{x}}_{N}^{m},\boldsymbol{\xi})\right]$ and $\widehat{J}_{N}^{m}$ are random variables (they are direct functions of the sample data), we conduct a certain number of runs, each with an independent sample of size $N$. This way we can provide (visual) estimates of the expected value and variability of the out-of-sample performance and disappointment for several values of the sample size $N$. These estimates are illustrated in the form of box plots in a series of figures. In these figures,  the dotted black horizontal line corresponds to either  solution $x^*$ or to its associated optimal cost $J^*$ with complete information (i.e., without ambiguity about the true data distribution).

For the sole purpose of conducting a comparison as fairly as possible, parameters $\varepsilon$ and $\rho$ in both DROC and DROW are tuned so that the underlying true distribution of the data belongs to the corresponding ambiguity set with, at least, a pre-fixed confidence level of probability.  In the case of the newsvendor examples, we guarantee this by trial and error for simplicity. In practice, however,  these parameters should be calibrated by way of a (statistical) procedure that uses the data available to the decision maker, for example, through cross-validation or bootstrapping. We follow this approach in the problem of the Cournot producer.
Finally, we stress that, in our approach, caution should be exercised when selecting $\varepsilon$ and $\rho$, as they should be such that problem (P) has at least one feasible solution. This is not  guaranteed  in the case that the empirical distribution $\widehat{Q}$ does not satisfy the order cone constraints on the probability masses associated with each subregion $\Xi_i$ of the support set $\Xi$. Intuitively, in this case, optimization problem (P) must have enough ``budget'' (i.e., $\varepsilon$ and $\rho$ must be high enough)  to ``transport'' the empirical distribution to another one that complies with the a-priori information. In other words, the ambiguity set of problem (P) must be sufficiently large to contain at least one probability distribution that assigns probability masses verifying the order cone constraints to the partitions.

\end{enumerate}
\subsection{Newsvendor problems}
In this section, we illustrate the theoretical results of our paper on the popular \emph{newsvendor} problem (also known as
the \emph{newsboy problem}). Many extensions and variants of this problem have been considered since it was first posed in the 50s (see, for example, the work in \cite{gallego1993},  \cite{Choi2012}, \cite{Andersson2013}, \cite{Pando2014}, and references therein).
According to \cite{Pando2013},

\emph{The newsboy problem is probably the
most studied stochastic inventory model in inventory control theory and the
one with most extensions in recent years. This problem reflects many real-life situations and is often used to aid decision making in both
manufacturing and retailing. It is particularly important for items with
significant demand uncertainty and large over-stocking and under-stocking
costs.}

\subsubsection{The single-item newsvendor problem}\label{single-item}
In the single-item newsvendor model, the decision maker has to plan the
inventory level for a certain product before the random
demand $\xi$ for that product is realized, facing both holding and backorder  costs.
The newsvendor problem can be formulated as
$$\inf_{x \geqslant 0} \mathbb{E}_Q [ h(x-\xi)^+ + b(\xi-x)^+
]$$
where $x$ is the order quantity, and $b,h>0$ are the unit holding cost and the unit backorder cost, respectively. Here we have assumed that $h=4$  and $b=2$.

The demand for the item (unknown to the decision maker) is assumed to follow a mixture (with weights $\omega_1=0.1,\; \omega_2=0.35$ and $\omega_3=0.55$) of the three
normal distributions  $\mathcal{N}_1(0.2,0.05),\;\mathcal{N}_1(0.5,0.1),\;$ and $\mathcal{N}_1(0.8,0.05)\;$, truncated over the unit interval $[0,1]$. Figure~\ref{plot_distr_univ_real.pdf} provides a visual illustration of the resulting mixture. Recall that, in the numerical experiments that follow, we have used 15\,000 samples drawn from this mixture of Gaussian distributions to approximate the true distribution of the item demand and to partition its support set $[0,1]$ into four regions, based on the two-phase procedure we have previously described. In fact, what we show in Figure~\ref{plot_distr_univ_real.pdf} is the histogram of those 15\,000 data points and its corresponding kernel density estimate.

For the sole purpose of conducting a comparison as fairly as possible,
parameters $\varepsilon$ and $\rho$ in both DROC and DROW are tuned so that the underlying true distribution
of the data belongs to the corresponding ambiguity set with
at least $95\%$ of probability. We check whether this condition holds or not \emph{a posteriori} (by trial and error), by counting the number of
runs (out of the one thousand we perform) for which the
out-of-sample disappointment
is negative.

The values we have used for the parameters $\varepsilon$ and $\rho$ in DROC and DROW are collated in Table~\ref{tab:radios_univ}. We insist that these parameters have been adjusted so that at most 50 out of the 1000 runs we have conducted for each sample size $N$ deliver a positive out-of-sample disappointment (that is, to achieve and maintain a similar level of reliability for the data-driven solutions given by DROC and DROW). As expected, therefore, the values of both $\varepsilon$ and $\rho$ decrease as the sample size $N$ grows.

\begin{table}[ht!]
\caption{Single-item newsvendor problem: Values for parameters $\varepsilon,\rho$ in DROC and $\rho$ in DROW.}
\label{tab:radios_univ}
\centering
\setlength\tabcolsep{3pt} 

\begin{tabular}{ccccc}
\hline
\multirow{2}{*}{$N$} &  \multicolumn{2}{c}{DROC}& DROW \\
\cline{2-3}
 &  $\varepsilon$ & $\rho$ & $\rho$ &\\
\hline
2  &  0.9 &  0.9 &  1&  \\
5  &  0.8 &  0.8 &  0.9&  \\
10  &  0.7 &  0.7&  0.8&  \\
20  &  0.4 &  0.6&  0.6&  \\
50  &  0.15 &  0.25&  0.4&  \\
100  &  0.1 &  0.2&  0.25&  \\
200  &  0.01 &  0.15&  0.05&  \\
\hline
\end{tabular}
\end{table}

Figures~\ref{optimal_order_quantity1_univariate_ray}, \ref{out-of-sample_univariate_ray}, and \ref{actual_expected_cost_univariate_ray} show the box plots corresponding to the order quantity, the out-of-sample disappointment and the actual expected cost delivered by each of the considered data-driven approaches for various sample sizes. The  shaded areas have been obtained by joining the whiskers of the box plots, while the associated solid lines link their medians. Interestingly, whereas the medians of the order quantity estimators provided by SAA are very close to the optimal one $x^*$, their high variability results in (large) disappointment with very high probability. On the contrary, the median of the order quantity delivered by DROW is significantly far from the optimal one (with complete information) for small sample sizes, but it manages to keep the out-of-sample disappointment below zero in return. To do so, however, DROW tends to produce costly (overconservative) solutions on average, as inferred from their actual expected cost in Figure~\ref{actual_expected_cost_univariate_ray}. In plain words, DROW pays quite a lot to ensure a highly reliable/robust order quantity. The proposed approach DDRO, however, is able to leverage the a-priori information on the partition probabilities $(p_i)_{i=1}^{|\mathcal{I}|}$ to substantially reduce the cost to pay for reliable data-driven solutions, especially for small sample sizes. Intuitively, this information enables DROC to identify highly reliable solutions that are myopically deemed as non-reliable and, therefore, discarded by DROW. Logically, this is contingent on the quality of the a-priori information that is supplied to DROC in the form of order  cone constraints on $(p_i)_{i=1}^{|\mathcal{I}|}$.

\begin{figure}[htp]
\centering

\subfloat[Data generating distribution (kernel density estimate)]{%
  \includegraphics[width=0.5\textwidth]{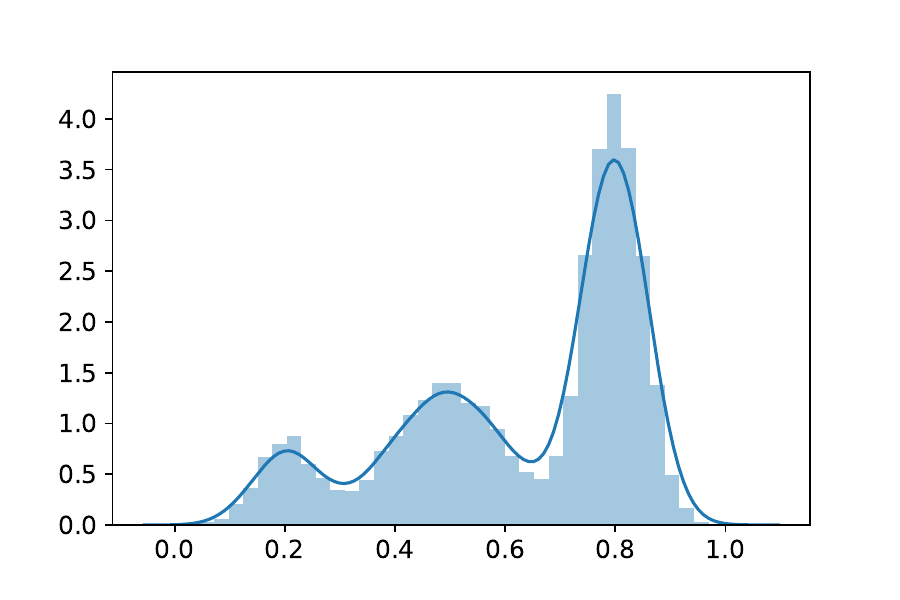}%
  \label{plot_distr_univ_real.pdf}
}%
\subfloat[Order Quantity]{%
  \includegraphics[width=0.5\textwidth]{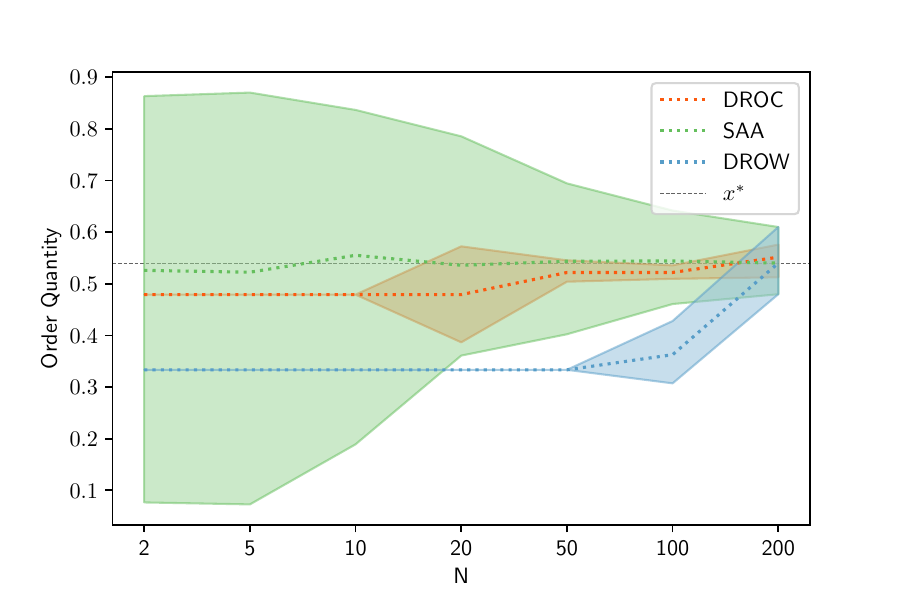}%
  \label{optimal_order_quantity1_univariate_ray}
}

\subfloat[Out-of-sample disappointment]{%
  \includegraphics[width=0.5\textwidth]{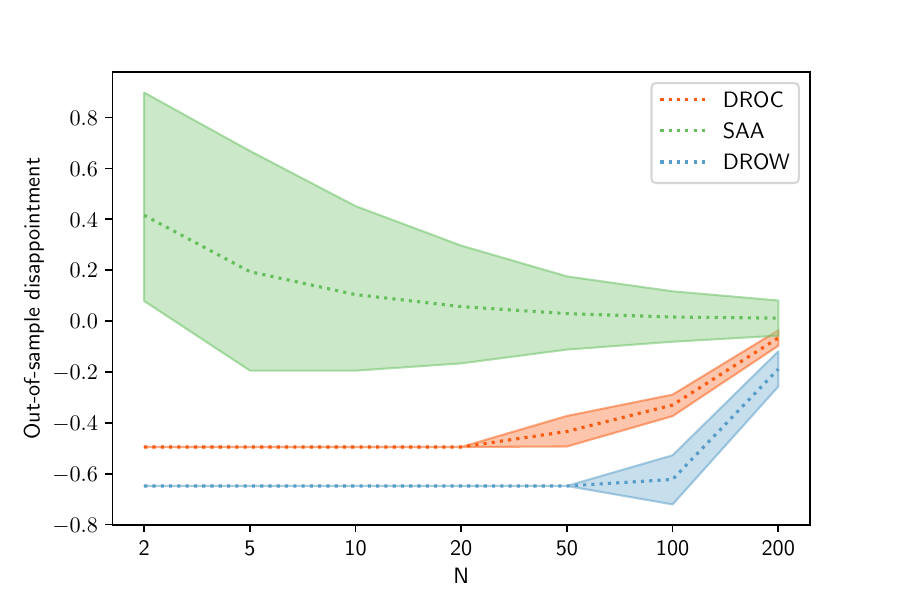}%
  \label{out-of-sample_univariate_ray}
}%
\subfloat[Actual expected cost]{%
  \includegraphics[width=0.5\textwidth]{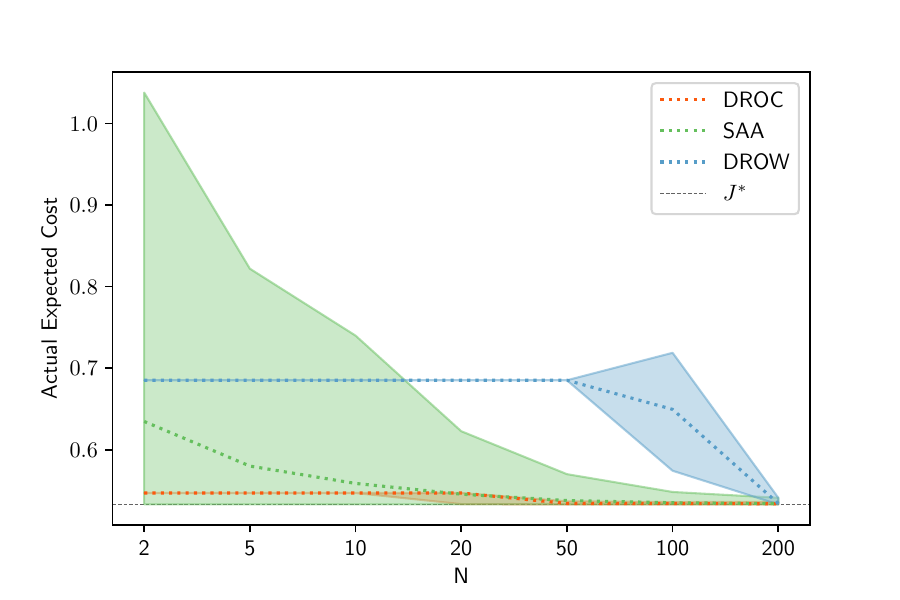}%
  \label{actual_expected_cost_univariate_ray}
}

\vspace{5mm}

\caption{Single-item newsvendor problem: (Approximate) true data-generating distribution, order quantity and performance metrics}

\end{figure}

\subsubsection{The multi-item newsvendor problem}
In this section, we carry out an analysis similar to that of Subsection~\ref{single-item}, but for the multi-item newsvendor problem, which can be formulated as follows:
$$\inf_{\mathbf{x} \geqslant \mathbf{0}} \mathbb{E}_Q \sum_{l=1}^d[ h_l(x_l-\xi_l)^+ + b_l(\xi_l-x_l)^+ ]$$
where $x_l$ is the order quantity for the $l$-th item, $Q$ is the joint probability distribution governing the demands for the $d$ items, and $b_l,h_l>0$ are the unit holding cost and the unit backorder cost for the $l$-th item, respectively.

To  illustrate our approach in a higher dimensional setting, we consider twenty items, i.e., $d = 20$.

We consider the following parameters: $h_1=\ldots =h_{10}=2, h_{11}=\ldots=h_{20}=4$, $b_1=\ldots =b_{10}=4$; and $b_{11}=\ldots=b_{20}=2$.
The demands for the twenty items are assumed to follow a mixture of three multivariate normal distributions $\mathcal{N}_{20}(\boldsymbol{\mu}_1,\Sigma_1
),$ $\;\mathcal{N}_{20}(\boldsymbol{\mu}_2,\Sigma_2 ),\;$ and
$\mathcal{N}_{20}(\boldsymbol{\mu}_3,\Sigma_3 ),\;$ where
$\boldsymbol{\mu}_1=[3,\ldots,3] \in \mathbb{R}^{20},\ \Sigma_1=\textrm{diag}(1,\ldots,1) \in \mathbb{R}^{20\times 20}$;
$\boldsymbol{\mu}_2=[5,\ldots,5]\in \mathbb{R}^{20},\ \Sigma_2=\textrm{diag}(0.5,\ldots,0.5)\in \mathbb{R}^{20\times 20}$; and
$\boldsymbol{\mu}_3=[7,\ldots,7] \in \mathbb{R}^{20},\ \Sigma_3=\textrm{diag}(0.1,\ldots,0.1)\in \mathbb{R}^{20\times 20}$. The weights of the mixture are $\omega_1=0.1,\; \omega_2=0.65$ and $\omega_3=0.25$, respectively. Furthermore, the mixture has been truncated on the hypercube $[0,10]^{20}$.

The values we have used for the parameters $\varepsilon$ and $\rho$ in DROC and DROW are collated in Table~\ref{tab:radios_multiitem20}.

\begin{table}[ht!]
\caption{Multi-item newsvendor problem: Values for parameters $\varepsilon$, $\rho$ in DROC and $\rho$ in  DROW}
\label{tab:radios_multiitem20}
\centering
\setlength\tabcolsep{3pt} 
\begin{tabular}{ccccc}
\hline
\multirow{2}{*}{$N$} &  \multicolumn{2}{c}{DROC}& DROW \\
\cline{2-3}
 &  $\varepsilon$ & $\rho$ & $\rho$ &\\
\hline
2  &  5 &  2 &  60&  \\
5  &  5 &  2 & 50&  \\
10  &  4.5 &  1.5&  40&  \\
20  &  4 &  1&  20&  \\
50  &  2.5&  0.6&  10&  \\
100  &  1.75 &  0.5&  8&  \\
200  &  1.25 &  0.35&  4&  \\
\hline
\end{tabular}
\end{table}

Again, for a meaningful and fair comparison, these parameters have been tuned by trial and error in such a way that at most 50 out of the 1000 runs we have carried out for each sample size $N$ yield a positive out-of-sample disappointment. The values for the parameters, which we need to this end, diminish as we gain more information (i.e., as the sample size $N$ grows). Note that, for small sample sizes, for which the available data provide very little information about their true distribution, a great deal of robustness is required to produce highly reliable data-driven solutions. Consequently, it is little wonder that the selected values for $\rho$ in DROC are equal to two, which is the maximum value that the total variation distance between $P$ and $\widehat{P}$ can take on.
\begin{figure}[htp]
\centering

\subfloat[Out-of-sample disappointment]{%
  \includegraphics[width=0.5\textwidth]{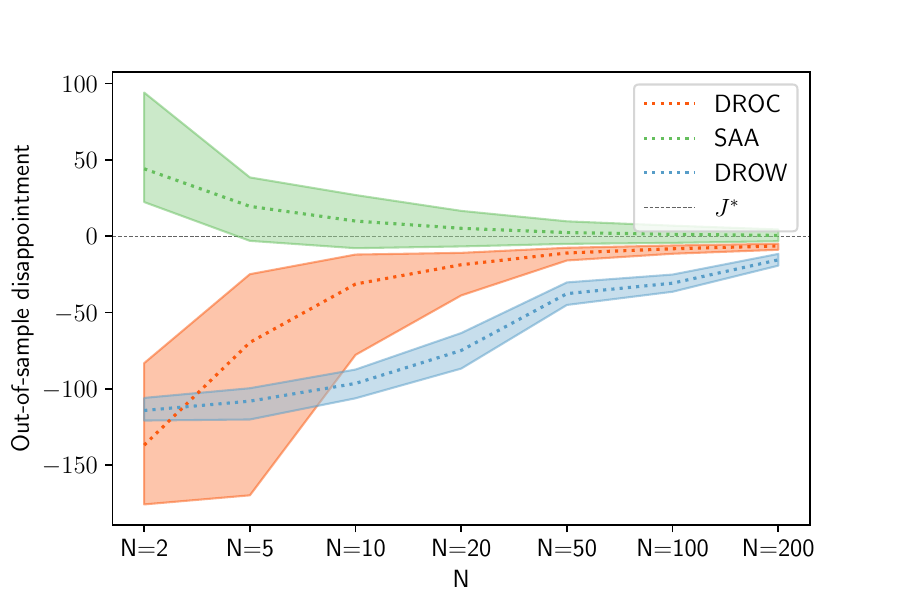}%
  \label{out-of-sample_multiitem_ray}
}%
\subfloat[Actual expected cost]{%
  \includegraphics[width=0.5\textwidth]{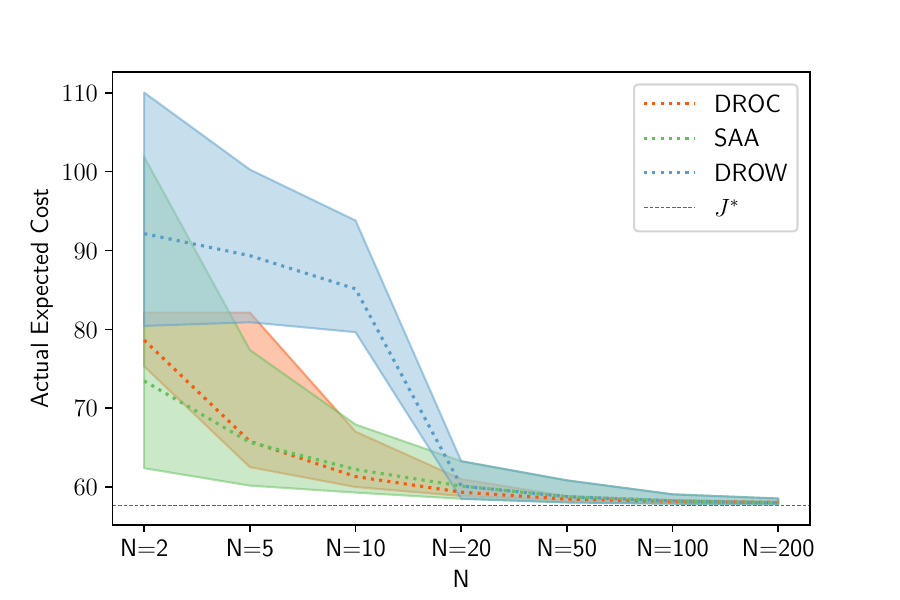}%
  \label{actual_expected_cost_multiitem_ray}
}
\vspace{5mm}

\caption{Multi-item newsvendor problem: Performance metrics}

\end{figure}
In the same fashion as in the case of the previous example of the single-item newsvendor problem, Figures~ \ref{out-of-sample_multiitem_ray} and \ref{actual_expected_cost_multiitem_ray} show, for various sample sizes, the box plots pertaining to the  out-of-sample disappointment and the actual expected cost associated with each of the considered data-driven approaches, in that order. The results conveyed by these figures confirm our initial conclusions: The ability of our approach DROC to exploit a-priori knowledge of the order among some partition probabilities permits identifying solutions that perform noticeably better out of sample with the same level of confidence. We underline that, in terms of the out-of-sample disappointment, the decision maker seeks a data-driven method $m$ that renders an estimate $\widehat{J}_{N}^{m}$ that results in a positive surprise (i.e., negative disappointment) with a high probability, but that is as close as possible to the cost with full information $J^*$. Consequently, the large  negative out-of-sample disappointment that the solutions given by DROW feature can be attributed to its over-conservativeness.

In terms of computational time, solving DROC for this instance of the multi-item newsvendor problem, with 20 items, four partitions and a sample size of 200, takes less than a second with CPLEX 12.10 running on a Windows 10 PC with a CPU Intel (R) Core i7-8550U clocking at 1.80 GHz and 8 GB of RAM.

\subsection{The problem of a strategic firm competing \`a la Cournot in a market}

Next we consider the problem of a strategic firm competing \`a la Cournot in a market for an undifferentiated product. This could be the case of, for instance, the electricity market (see, e.g., \cite[Ch. 3]{gabriel2012complementarity} and  \cite{Ruiz2008}). Suppose the firm can produce up to one per-unit amount of product at a cost given by $a_{2} x^{2} + a_{1} x + a_{0}$, where $x$ is the per-unit amount of product eventually produced and $a_{0}, a_{1}$ and $a_{2}$ are \emph{known} parameters taking values in $\mathbb{R}^+$. Furthermore, assume an inverse residual demand function in the form $\lambda = \alpha-\beta x$, where $\lambda$ is the market clearing price for the product, and $\alpha, \beta \in \mathbb{R}^+$ are \emph{unknown and uncertain} parameters. The firm seeks, therefore, to minimize its cost $(a_{2} x^{2} + a_{1} x + a_{0})-\lambda x$ subject to $x \in[0,1]$. After some basic manipulation, the problem of the firm can be posed as
$$\inf_{x \in [0,1]} \mathbb{E}_Q [(-x)\xi+x^2]$$
where $\xi = \frac{\alpha-a_1}{\beta +a_2}$.

The most interesting feature of this example is that, unlike in the aforementioned newsvendor problems, the Lipschitz constant of the objective function $f(x,\xi):= (-x)\xi+x^2$ with respect
to $\boldsymbol{\xi}$ is dependent  on  the  decision  variable $\mathbf{x}$.

We consider that $\xi$ follows a (true) probability distribution given by 10\,000 points sampled from a mixture of three Gaussian distributions with variances all equal to $0.3$ and means $\mu_1=0,\mu_2=1.2$ and $\mu_3=2.5$. The weights of the mixture are $\omega_1=0.5,\omega_2=0.2$ and $\omega = 30.3$. Furthermore, the mixture has been truncated (over the interval $[-1.8,3]$. Figure~\ref{muestra_productor_ordercone} plots the kernel estimate of the data-generating distribution.

As in the previous experiments, we have divided the support $[-1.8,3]$ into four partitions, using the procedure described at the beginning of Section~\ref{numerics}. However, in a different way to what we did in  the newsvendor examples, here we select parameters $\varepsilon$ and $\rho$ following a procedure that solely relies on the available data, similarly to what is done in \cite{MohajerinEsfahani2018}. Essentially, given a desired confidence level $(1-\beta)$ for the finite-sample guarantee (set to 0.85 in our numerical experiments), we need to \emph{estimate}, using the data sample available only,  the parameters $\varepsilon$ and $\rho$ that deliver, at least, this confidence level while yielding the best out-of-sample performance. To this end, we use bootstrapping. The estimator of those parameters is denoted as  $param_{N}^m(\beta)$, underlining that the number and type of parameters to be estimated depend on the method $m$. The estimation procedure is carried out as follows for each sample of size $N$ (in this experiment, we consider $300$ independent data samples for each size $N$):
\begin{enumerate}
    \item  We construct $kboot$ resamples of size $N$ (with replacement), each playing the role of a different training dataset. Moreover, take those data points that have not been resampled to form a validation dataset (one per resample of size $N$). In our experiments below, we have considered $kboot = 50$.

    \item For each resample $k=1,\ldots,kboot$ and each candidate value for $param$, get a DRO solution from method $j$ with parameter (or pair of paramaters) $param$ on the $k$-th resample. The resulting optimal decision  is denoted as $\widehat{x}^{j,k}_N(param)$ and its associated objective value as $\widehat{J}^{j,k}_N(param)$. Subsequently, we compute the out-of-sample performance $J(\widehat{x}^{j,k}_N(param))$  of the data-driven solution $\widehat{x}^{j,k}_N(param)$ over the  $k$-th validation dataset.
\item From among the candidate values for $param$ such that $\widehat{J}^{j,k}_N(param)$
exceeds the value $J(\widehat{x}^{j,k}_N(param))$
in at least
$(1-\beta)\times kboot$ different resamples, take the one with the lowest $\frac{\sum_{k=1}^{kboot}{J(\widehat{x}^{j,k}_N(param))}}{kboot}$ (that is, with the highest out-of-sample performance averaged over the $kboot$ resamples).
\item Finally, compute the solution given by method $j$ with parameter $param^{\beta,j}_N$, $\widehat{x}^j_N:=\widehat{x}^{j}_N(param^{\beta,j}_N)$ and the respective certificate
$\widehat{J}^j_N:=\widehat{J}^{j}_N(param^{\beta,j}_N)$.
\end{enumerate}

\begin{figure}[htp]
\centering

\subfloat[Data generating distribution (kernel density estimate)]{%
  \includegraphics[width=0.5\textwidth]{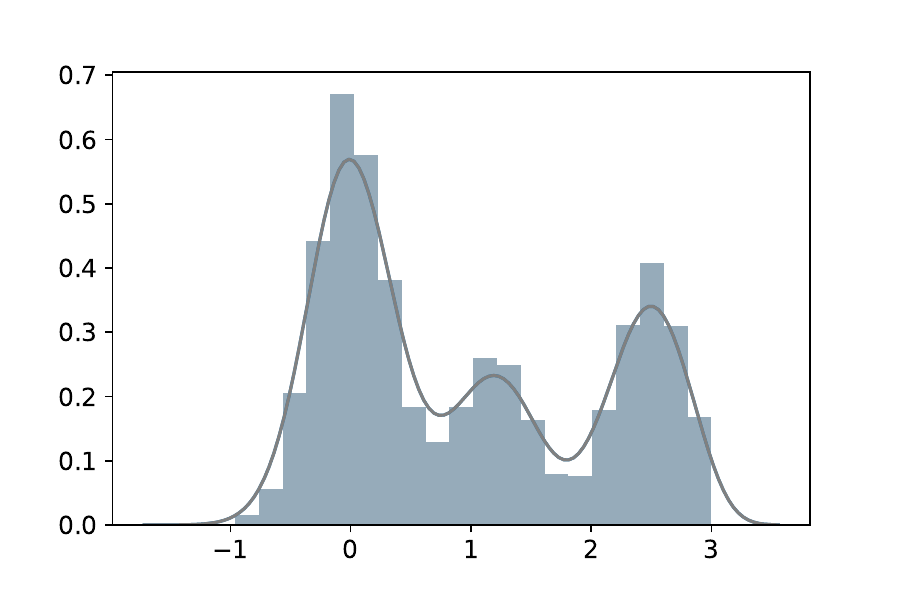}%
  \label{muestra_productor_ordercone}
}%
\subfloat[Optimal solution]{%
  \includegraphics[width=0.5\textwidth]{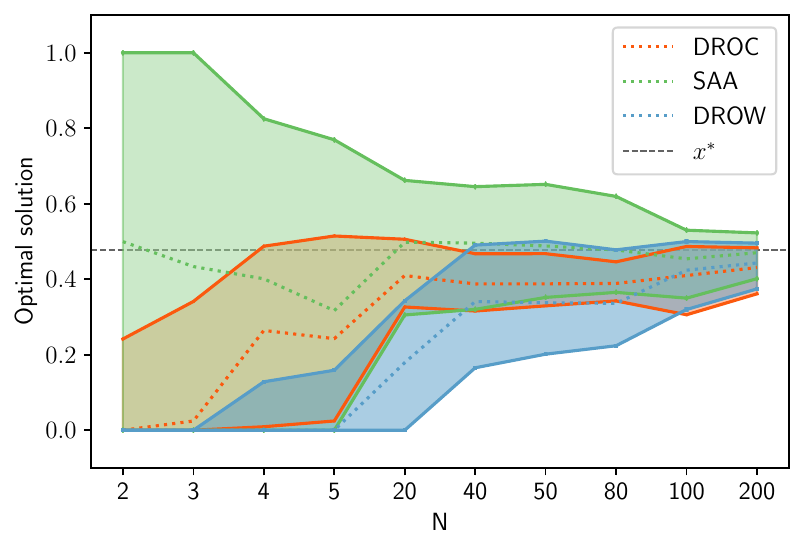}%
  \label{xsol_univariate_ray_prod_cone}
}

\subfloat[Out-of-sample disappointment]{%
  \includegraphics[width=0.5\textwidth]{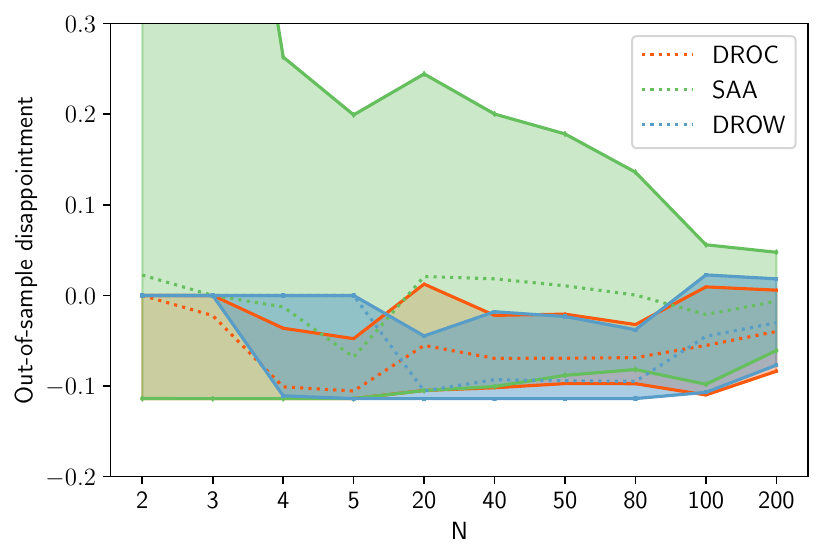}%
  \label{out-of-sample_univariate_ray_prod_cone}
}%
\subfloat[Actual expected cost]{%
  \includegraphics[width=0.5\textwidth]{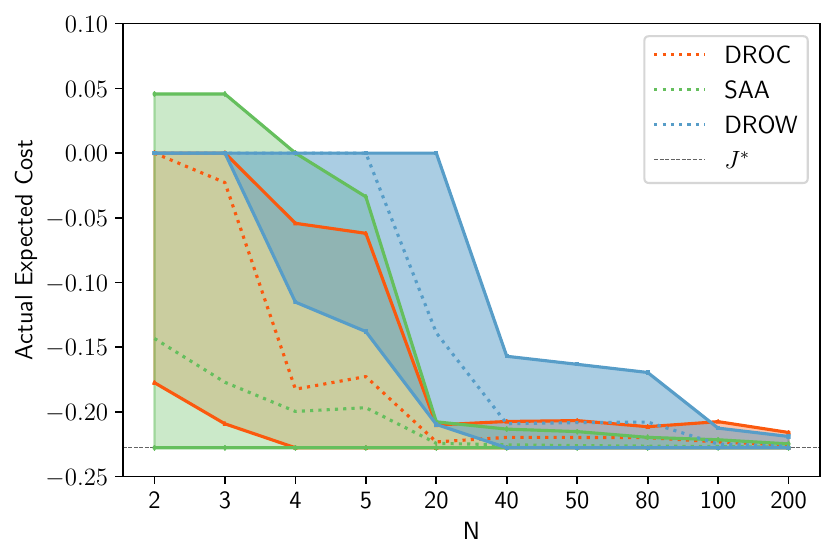}%
  \label{actual_expected_cost_univariate_ray_prod_cone}
}

\vspace{5mm}

\caption{Strategic firm problem: (Approximate) true data-generating distribution, optimal solution and performance metrics}

\end{figure}

As for the newsvendor examples, Figures~\ref{xsol_univariate_ray_prod_cone},
\ref{out-of-sample_univariate_ray_prod_cone},
and \ref{actual_expected_cost_univariate_ray_prod_cone} show, for various sample sizes, the box plots pertaining to the optimal decision, the out-of-sample disappointment and the actual expected cost associated
with each of the considered data-driven approaches, in that order. Once again, the results conveyed by these figures confirm our previous  conclusions: Our approach DROC is able to leverage a-priori knowledge of the order among some partition
probabilities to deliver solutions that perform significantly
better out of sample for the same level of
confidence. Furthermore, we see that the decision computed by the proposed
method DROC converges to the true optimal solution (with complete information) faster than the solutions provided by the other methods.

\section{Conclusions}\label{conclusions}
In this paper, we have presented a novel framework for data-driven distributionally robust optimization (DRO) based on optimal transport theory in combination with order cone constraints to leverage \emph{a-priori} information on the true data-generating distribution.
Motivated by the reported  over-conservativeness of the traditional DRO approach based on the Wasserstein metric, we have formulated an ambiguity set able to incorporate information about the order among the probabilities that the true distribution of the problem's uncertain  parameters assigns to some subregions of its support set. Our approach can accomodate a wide range of shape information (such as that related to monotonicity or multi-modality) in a practical and intuitive way. Moreover, under mild assumptions, the resulting distributionally
robust  optimization  problem  can  be,  in  fact,  reformulated  as  a  finite  convex  problem  where  the
a-priori
information (expressed through the order cone constraints) are cast as linear constraints as opposed to the more computationally challenging
formulations that exist in the literature. Furthermore, our approach is supported by theoretical performance guarantees and is capable of turning the provided information into solutions with increased reliability and improved performance, as illustrated by the numerical experiments we have prepared based on the well-known newsvendor problem and the problem of a strategic firm competing \'a la Cournot in a market for a homogeneous product.

\section*{Compliance with ethical standards}

\textbf{Funding}\hspace{3mm} This research has received funding from the European Research Council (ERC) under the European Union's Horizon 2020 research and innovation programme (grant agreement no. 755705). This work was also supported in part by the Spanish Ministry of Economy, Industry and Competitiveness and the European Regional Development Fund (ERDF) through project ENE2017-83775-P.\\

\noindent \textbf{Conflict of interest}\hspace{3mm} The authors declare that they have no conflict of interest.


%
%


%
%

%
%

\end{document}